\documentclass[11pt]{article}

\usepackage[utf8]{inputenc}
\usepackage[authoryear]{natbib}
\usepackage{amsmath} 
\usepackage{amsfonts}
\usepackage{mathrsfs}
\usepackage{amssymb}
\usepackage{enumitem}
\usepackage{amsthm}
\usepackage[none]{hyphenat}
\usepackage{amsthm} 
\usepackage{graphicx,psfrag}
\usepackage{sidecap} 
\usepackage[T1]{fontenc}
\usepackage{caption}
\usepackage{subfig}
\usepackage{fix-cm}
\usepackage{geometry}
\usepackage{titling} 
\usepackage{titlesec}
\usepackage{sectsty} 
\usepackage{mathtools}
\usepackage{accents}

\usepackage{fancyhdr} 
\newcommand{\HorRule}{\rule{\linewidth}{.2pt}} 

\pretitle{\vspace{-30pt} \begin{flushleft} \fontsize{20}{20} \usefont{OT1}{phv}{b}{n} \selectfont} 

\title{Article Title} 

\posttitle{\par \end{flushleft}\vskip 0.5em} 

\preauthor{\begin{flushleft}\large \lineskip 0.5em \usefont{OT1}{phv}{m}{n} } 

\author{John Smith, } 

\postauthor{ \\[1em] \fontsize{8}{8} \usefont{OT1}{phv}{m}{n}  Faculty of Technology, Bielefeld University, Box 100131, 33501 Bielefeld, Germany
\\[.4em]
\normalsize \usefont{OT1}{phv}{m}{n}

\par \vskip -.5em \end{flushleft} \HorRule} 

\def\abstract{\topsep=0pt\partopsep=0pt\parsep=0pt\itemsep=0pt\relax
\trivlist\item [\hskip\labelsep
{\usefont{OT1}{phv}{b}{n} \abstractname}]\if!\abstractname!\hskip-\labelsep\fi }

\providecommand{\keywords}[1]{\usefont{OT1}{phv}{b}{n}  Keywords: \normalfont #1}
\providecommand{\subclass}[1]{MSC 2010 Subject classification: \normalfont #1}

\allsectionsfont{\large \usefont{OT1}{phv}{b}{n}} 
\paragraphfont{\normalsize \usefont{OT1}{phv}{b}{n}}
\subsectionfont{\normalsize \usefont{OT1}{phv}{b}{n}}
\makeatletter
\let\origsection\section
\renewcommand\section{\@ifstar{\starsection}{\nostarsection}}

\newcommand\nostarsection[1]
{\sectionprelude\origsection{#1}\sectionpostlude}

\newcommand\starsection[1]
{\sectionprelude\origsection*{#1}\sectionpostlude}

\newcommand\sectionprelude{%
  \vspace{1em}
}

\newcommand\sectionpostlude{%
  \vspace{-.7em}
}
\makeatother

\theoremstyle{plain}
\newtheorem{lemma}{Lemma}[]
\newtheorem{coro}{Corollary}[]
\newtheorem{theorem}{Theorem}[]
\newtheorem{proposition}{Proposition}[]
\newtheorem{fact}{Fact}[]

\makeatletter
\newcommand{\pushright}[1]{\ifmeasuring@#1\else\omit\hfill$\displaystyle#1$\fi\ignorespaces}
\makeatother
\theoremstyle{definition} 
	\newtheorem{definition}{Definition}[]
	\newtheorem{example}{Example}[]
	
	\newtheorem{remark}{Remark}[]

\makeatletter
\renewcommand{\pushright}[1]{\ifmeasuring@#1\else\omit\hfill$\displaystyle#1$\fi\ignorespaces}
\makeatother

\newcommand{\0}{\mathbf{0}}
\newcommand{\1}{\mathbf{1}}
\newcommand{\A}{\mathcal A}
\newcommand{\B}{\mathcal B}
\newcommand{\C}{\mathcal C}
\newcommand{\D}{\mathcal{D}}
\newcommand{\Ee}{\mathcal{E}}
\newcommand{\E}{\mathbf{E}}

\newcommand{\J}{\mathcal{J}}
\newcommand{\K}{\mathcal{K}}
\newcommand{\M}{\mathcal{M}}
\newcommand{\OO}{\mathcal{O}}
\renewcommand{\P}{\mathbf P}

\newcommand{\XX}{\mathbb{X}}

\newcommand{\CC}{\mathbb{C}}
\newcommand{\PP}{\mathbb{P}}
\newcommand{\OP}{\mathbb{O}}

\newcommand{\<}{\preccurlyeq}
\renewcommand{\>}{\succcurlyeq}
\newcommand{\ldot}{\lower.5ex\hbox{\ensuremath{\cdot}}}
\newcommand{\ts}{\hspace{0.5pt}}
\newcommand{\tts}{\hspace{1pt}}
\newcommand{\nts}{\hspace{-0.5pt}}
\newcommand{\pa}{\hphantom{g}\nts\nts}
\newcommand{\bigtim}{\mbox{\Large $\times$}}
\DeclareSymbolFont{largesymbolsA}{U}{txexa}{m}{n}
\DeclareMathSymbol{\varprod}{\mathop}{largesymbolsA}{16}
\DeclareMathAlphabet{\pazocal}{OMS}{zplm}{m}{n}

\renewcommand{\H}[2]{\emptybar H_{\! #1}^{\; #2}}
\newcommand{\Hbar}[2]{\xxbar H_{\! \ts #1}^{\; #2}}
\newcommand{\R}[2]{\emptybar R_{#1}^{\,\ts #2}}
\newcommand{\Rbar}[2]{\xxbar R_{#1}^{\; #2}}
\renewcommand{\L}[2]{\emptybar L_{#1}^{\ts #2}}
\newcommand{\varThet}[2]{\varTheta_{\! #1}^{#2}}
\newcommand{\varThetp}[1]{\varTheta_{\! #1}'}
\newcommand{\varThetpp}[1]{\varTheta_{\! #1}''}
\newcommand{\ra}{r^{}_{\nts \A}}
\renewcommand{\r}[2]{ r_{\! #1}^{\, #2}}
\renewcommand{\o}[2]{\varrho_{#1}^{\, #2}}
\newcommand{\underdot}[1]{\underset{\accentset{\vspace{.5ex}\mbox{\large .}}{}}#1}

\newcommand{\sumsubstack}[2]{\underset{\substack{#1 \\[-1mm] #2}}{\sum}}
\newcommand{\Ncard}[1]{N^{\scriptscriptstyle{\lvert #1 \rvert}}}
\newcommand{\Ncardd}[2]{N^{\scriptscriptstyle{\lvert #1 \rvert - \lvert #2 \rvert}}}

\newcommand*\xxbar[1]{%
  \hbox{%
   \:\! \!\vbox{%
      \hrule height 0.5pt 
      \kern0.25ex
      \hbox{%
        \kern -0.2 em
        \ensuremath{#1}%
        \kern-0.1em%
         }}}} 

\newcommand*\emptybar[1]{%
  \hbox{%
   \:\! \!\vbox{%
      \hrule height 0.0pt 
      \kern0.25ex
      \hbox{%
        \kern -0.2 em
        \ensuremath{#1}%
        \kern-0.1em%
         }}}} 

\setitemize{leftmargin=50pt}  

\bibpunct[, ]{(}{)}{;}{a}{}{;}

\begin{document}


\title{Partitioning, duality, and linkage disequilibria in the Moran model
with recombination}
\author{Mareike Esser \thanks{E-mail: messer@techfak.uni-bielefeld.de} \;\,-- \,Sebastian Probst \thanks{E-mail: sprobst@techfak.uni-bielefeld.de} \;\,-- \,Ellen Baake \thanks{E-mail: ebaake@techfak.uni-bielefeld.de} }
\date{}

\maketitle

\thispagestyle{fancy} 

%
\begin{small}
\abstract{
The multilocus Moran model with recombination is considered, which describes the evolution of the genetic composition of a population under recombination and resampling. We investigate a marginal ancestral recombination process, where each site is sampled only in one individual and we do not make any scaling assumptions in the first place. Following the ancestry of these loci backward in time yields a partition-valued Markov process, which experiences splitting and coalescence. In the diffusion limit, this process turns into a marginalised version of the multilocus ancestral recombination graph. With the help of an inclusion-exclusion principle and so-called recombinators we show that the type distribution corresponding to a given partition may be represented in a systematic way by a sampling function. The same is true of correlation functions (known as linkage disequilibria in genetics) of all orders. \\
We prove that the partitioning process (backward in time) is dual to the Moran population process (forward in time), where the sampling function plays the role of the duality function. This sheds new light on the work of \citet{Polen}. The result also leads to a closed system of ordinary differential equations for the expectations of the sampling functions, which can be translated into expected type distributions and expected linkage disequilibria.
}

\keywords{Moran model with recombination, ancestral recombination process, linkage disequilibria, Möbius inversion, duality}

\subclass{92D10, 60J28.}

\end{small}

\section{Introduction}\label{sec:introduction}
Models that describe the evolution of finite populations
under recombination are among the major challenges in population genetics.
This article is devoted to the \emph{Moran model with recombination} (in continuous time),
which is briefly described as follows (see \citealt{Durrett,Polen}).
A chromosome  is identified with  a linear arrangement (or \emph{sequence})
of $n$ discrete positions called \textit{sites}, which are collected in the set $S=\{1,\dotsc,n\}$. A site may be understood as a nucleotide site or a gene locus. We will throughout
consider chromosomes as (haploid) \emph{individuals}, that is, we think at
the level of gametes (rather than that of  diploid individuals that
carry two copies of the genetic information). 
Site $i$ is occupied by letter $x_i \in \XX^{}_i$, where  $\XX_i$ is a finite set, $1\leqslant i \leqslant n$. If sites are nucleotide sites, a natural
choice for each $\XX_i$ is the nucleotide alphabet $\{\rm{A,G,C,T}\}$; 
if sites are gene loci,
$\XX_i$ is the set of alleles that can occur at locus $i$.
The genetic type of each individual is thus described by the sequence $ x=(x_1,x_2,\dotsc,x_n) \in \XX^{}_1 \times \dots \times \XX^{}_n =: \XX$, where
$\XX$ is the type space. Recombination means that a new individual
is formed as a `mixture' of an (ordered) pair of parents, say $x$ and $y$.
We will restrict ourselves to \emph{single-crossover recombination},
that is, the offspring inherits the leading segment (up to site $i$, for some
$1 \leqslant i < n$) from the first and
the trailing segment (after site $i$) from the second parent. The recombined
type thus is 
$( x_{\leqslant i}, y_{>i} ):= (x_1, \dotsc, x_i, y_{i+1}, \dotsc, y_n)$; we say
that a crossover has happened between sites $i$ and $i+1$. 
The sites that come from the paternal and the maternal sequence,
respectively, define a \emph{partition} ${\mathcal A}$ of $S$ into two
parts (we need not keep track
of which part was `maternal' and which was `paternal'). All partitions 
of $S$ into two ordered (or contiguous) parts 
($\A=\{\{1,2,\dotsc, i\},\{i+1, \dotsc, n\}\}, i \in S \setminus \{n\}$)
can be realised, via a single crossover event. 
Altogether, whenever an offspring is created, its sites are partitioned between parents according
to $\A$ with probability $\r{\A}{}$, where $\r{\A}{} \geqslant 0$,
$\sum_{\A \in \OP_2(S)}\r{\A}{} \leqslant 1$, and $\OP_2(S)$ is the set of all ordered
partitions of $S$ into two parts.
Let us note  that, due to the one-to-one
correspondence between elements of $S \setminus \{n\}$ and those of 
$\OP_2(S)$, the
specification of the $\r{\A}{}$ simply means that a crossover probability
is associated with each site in $S \setminus \{n\}$.
The sum $\sum_{\A \in \OP_2(S)}\r{\A}{}$ is the probability that some recombination
event takes place during reproduction. With probability
$\r{\ts\{S\}}{}=1- \sum_{\A \in \OP_2(S)}\r{\A}{}$, there is no recombination, 
in which case the offspring 
is the full copy of a single parent. 
We write $\OP_{\leqslant 2}(S):= \OP_2(S) \cup \{S\}$ for the set of ordered
partitions into at most two parts. The collection $\{\r{\A}{}\}_{\A \in \OP_{\leqslant 2}(S)}^{}$ is known as
the \emph{recombination distribution} \citep[p. 55]{Buerger}.

Consider now a \emph{population} of a constant number of $N$ haploid individuals (that is, gametes), which evolves as follows (see Figure~\ref{moran_realisation}). Each individual has an exponential lifespan with parameter $1$
(this choice of the parameter is without loss of generality; it simply sets
the time scale).
When an individual dies, it is replaced by a new one as follows.
First draw a partition $\A$ according to the recombination distribution. 
Then draw $|\A|$ parents from the population (the parents may
include the  individual that 
is about to die), uniformly and with replacement, where $|\A|$ is the number
of parts in $\A$. If $|\A|=2$, the offspring inherits the leading segment 
of $\A$ from the first and the trailing segment from the second parent, as
described above.  If $|\A|=1$ (and thus $\A=\{S\}$), 
the offspring is a full copy of a single parent 
(again chosen uniformly from among
all individuals); this
is called a \emph{(pure) resampling} event. All events
are independent of each other.

Note that it may seem biologically more realistic to draw two parents
\emph{without} replacement. However, assuming sampling \emph{with}
replacement entails significant simplifications, and yields the same
process as sampling without replacement with a slight change
in the recombination distribution. More precisely, since drawing the same
individual twice  means that the offspring is a full copy of this single
parent, our process agrees (in distribution) with the analogous process
without replacement if $\r{\A}{}$ is replaced by $\r{\A}{} (N-1)/N$ for all 
$\A \in \OP_2(S)$ (and $\r{\ts\{S\}}{}$ is set accordingly). 

\begin{figure}[ht]
\begin{center}
\includegraphics[width=.8\textwidth]{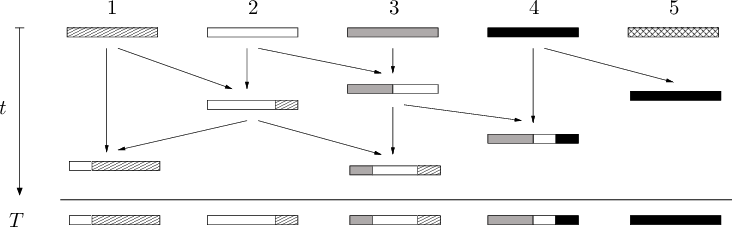}
\end{center}
\caption{\label{moran_realisation} 
Snapshot of a Moran model realisation with $N=5$ individuals. For example,
in the first event,  individual 3 dies and is replaced by a recombined  copy of individuals 2 and 3. The last line shows the  composition of the population
at the final timepoint,   $T$.}
\end{figure}

The model will be described more formally later. For now, let us summarise the two main lines of research in this context. On the one hand, there has been considerable interest in how the composition of the population evolves over time, and, in particular, how the correlations between sites (known as linkage disequilibria) will develop; see the overviews in
\citet[Chap.~5.4]{HeinSchierupWiuf}, \citet[Chap.~3.3 and 8.2]{Durrett}, or \citet[Chap.~7.2.4]{Wakeley}. 
Since there is no mutation, a single type will go to fixation
in the long run, that is, the 
entire population will ultimately consist of this single type. In the
absence of recombination, this will be one of the types initially present,
and it is well known that the fixation probability for a given type
equals its initial frequency.
If there is recombination, the type that ultimately wins
can also be a newly-composed type, but little is known about the 
fixation probabilities of the many possible types. The explicit development over
time is even more challenging, due to an intricate interplay of resampling
and recombination. It is usually approached forward
in time, e.g., \citet{OhtaKimura}, \citet{Polen3}, \citet{Polen4}, \citet{SongSong}, \citet{BaakeHerms},
\citet[Chap.\ 8.2]{Durrett}, or \citet{Polen}. In the deterministic limit, which emerges when $N \to \infty$ without rescaling of the $\r{\A}{}$ or of time, the population is described by a system of ordinary differential equations, again forward in time. This system has an explicit solution, both for the type distribution and for correlation functions of all orders, for an arbitrary number of sites \citep{BaakeBaake,Baake}. 
This also provides a decent approximation for large but finite populations 
\citep{BaakeHerms}, but dealing appropriately with the stochasticity
of finite populations remains a major challenge.

The second line of research is concerned with genealogical aspects and sampling formulae. Here, one starts from a sample taken from the present population and traces back the ancestry of the various segments the individuals are composed of.
A major challenge lies in the calculation of the probabilities for the type distribution of a random sample, that is, one aims at so-called \emph{sampling formulae}, see \citet[Chap.~3.6]{Durrett}. These questions are naturally approached backward in time. Usually, one employs the \emph{diffusion (or weak recombination) limit}, that is, time is sped up by a factor of $N$, followed by $N \to \infty$ such that $ N \r{\A}{} \to \o{\A}{}$, $\o{\A}{}$ a constant, $\A \in \OP_2(S)$. Obtaining sampling formulae is tied to the situation in which the population has reached a stationary state; even  this case is very hard to treat, and coping with time dependence seems to be hopeless. 

The aim of this article is to build a bridge between these two lines of research. We will explore the type distribution and the correlations over time, in the stochastic setting. A starting point will be a recent paper by \citet*{Polen}, who approach this question. Their setting is entirely forward in time, which effectively hides some of the underlying
structure. In contrast, we will proceed backward in time and provide a genealogocial approach for the analysis of correlations. The crucial notion in this context will be that of \emph{duality} between the original Moran model forward in time and a suitable ancestral process that follows back the ancestry of selected segments from today's population. This will also shed new light on the results of \citet{Polen}.
In order to keep the approach as general as possible, we will throughout adhere to the original (finite $N$) model, without taking any limit, but will discuss the various scalings and limits  where appropriate.

The paper is organised as follows. In Section~\ref{sec:prelim}, we start by
collecting some important facts about
partitions and M\"obius functions. We then (Section~\ref{sec:model}) introduce 
the model more formally and motivate our genealogical approach, which may
be considered a marginal version of the usual ancestral process with
recombination. In Section~\ref{sec:partitioning}, we describe our ancestral
process, which is a partitioning process that keeps track of how the
ancestral material is partitioned between individuals. In 
Section~\ref{sec:recombinators}, we introduce a
systematic description via recombinators, which describe the action of
recombination on a population and have
proved very useful in the deterministic setting. We 
complement them here by \emph{sampling functions}, which are additionally
required for finite populations. The collection of sampling functions will be
crucial since it will also serve as \emph{duality function} in
Section~\ref{sec:duality}, where the duality between the
Moran model forward in time and the partitioning process backward in time
is proved. This proof, at the same time, yields a differential equation
system for the expectations of the sampling functions, which are the
building blocks for the linkage disequilibria. 
In Section~\ref{sec:examples}, we apply our results to the cases of
two and three sites. We will see that the expected linkage disequilibria 
(of second and third order) decay exponentially even in the presence
of resampling, and identify further linear combinations of expected
sampling functions that decay exponentially.
For two sites,  we also obtain the explicit time course for the 
expected composition of the population, and, at the same time, the
fixation probabilities of the various types. 
%


\section{Preliminaries: Partitions and M\"obius functions}\label{sec:prelim}
Working with partitions will be essential to our approach, and we will
rely throughout on the powerful concept of \emph{M\"obius functions}
and \emph{M\"obius inversion}. Let us briefly collect the basic
notions and standard results;  more  background material as well as 
the proofs may be found in \citet{Rota}, \citet[Chap.~3.2]{Berge}, \citet[Chap.~I,II,IV]{Aigner} and
\citet[Chap.~3]{Stanley}. 
\paragraph{Partitions.} 
Let $W$ be a finite, nonempty, totally ordered set, such as a
finite subset of $\mathbb{N}$; later, $W$ will be $S$ or a 
subset thereof. Let $\PP = \PP (W)$ be the set of
partitions of $W$. We write such a partition as $\A = \{ A_{1}, \dots , A_{m}
\}$, where $A_{j}\neq \varnothing$ for all $j$ and
$A_{j} \cap A_{k} = \varnothing$ for all $j\ne k$ together with $A_{1} \cup
\dots \cup A_{m} = W$.
We call $A_j$ a \emph{block (or part)} of $\A$ and $m=|\A|$ is the number of blocks in $\A$.

We say that a partition $\A= \{ A_{1}, \dots , A_{m}\}$ of $W$ is
\emph{ordered} (or \emph{contiguous}, or an \emph{interval partition}) 
if  every  $A_j$ is ordered in $W$, that is,
$A_j = \{x \in W \mid \min A_j \leqslant x \leqslant \max A_j \}$. 
For example, if $W=\{1,2,5,7,9\}$, then $\{\{1,2,5\},\{7,9\}\}$ is
ordered, but  $\{\{1,2,7\},\{5,9\}\}$ is not.
The set of all ordered partitions of $W$ is denoted by $\OP(W)$, the set of
all ordered partitions of $W$ \emph{into (exactly) two parts} is $\OP_2(W)$, and the set of
all ordered partitions of $W$ \emph{into at most two parts} 
is $\OP_{\leqslant 2}(W)$.

For a given partition $\A= \{ A_{1}, \dots , A_{m}\}$ of $W$, let $M := \{ 1,2, \dotsc, m\} = M(\A)$ and, for $J \subseteq M$, we define
$\A_J := \{A_j\}_{j \in J}$ and $A_J := \cup_{j \in J} A_j$. $\A_J$ is a
partition of $A_J$. In particular,
$\A_M=\A$, $A_M = W$, $\A_{\{j\}} = \{A_j\}$, 
and $\A_{M \setminus \{j\}} = \A \setminus \{A_j\}$, for any
$j \in M$.
Note that $M$ depends on $\A$, but we suppress this dependence
when there is no risk of confusion. We will throughout abbreviate
$J \setminus j := J \setminus \{j\}$ and $J \cup k := J \cup \{k\}$. 

The natural ordering relation on $\PP(W)$ is denoted by
$\preccurlyeq$, where $\A \preccurlyeq \B$ means that $\A$ is a 
\emph{refinement} of $\B$, that is,  every block of $\A$ is a subset of
a block of $\B$; equivalently,  $\B$ is a \textit{coarsening} of $\A$.
$\A \prec \B$ means that $\A \preccurlyeq \B$ and  $\A \neq \B$. 
Together with the resulting  partial order,  $\PP(W)$  is a \textit{poset} and,
in particular, a \emph{finite lattice}.
$\PP(W)$ has a unique \emph{minimal} or \emph{finest} partition, which is
denoted as $\0 = \{ \{x\} \mid x \in W \}$; likewise, there
is a unique \emph{maximal} or \emph{coarsest} one, namely $\1 = \{ W \}$.

When $U$ and $V$ are \emph{disjoint} (finite) sets, two partitions
$\A\in\ts\PP(U)$ and $\B\in\ts\PP(V)$ can be joined into $\A \cup \B$ to form an
element of $\PP(U\, \dot\cup\, V)$. 
Furthermore, if $U\nts\subseteq W$, a partition $\A\in\ts\PP(W)$, with $\A = \{ A_{1},
\dots , A_{m} \}$ say, defines a unique partition of $U$ by
restriction. The latter is denoted by $\A|^{\pa}_{U}$, and its parts
are precisely all non-empty sets of the form $A_{i} \cap U$ with
$1\leqslant i \leqslant m$. In particular, $\1|^{\pa}_{U}$ is the coarsest element in $\PP(U)$.
For two partitions $\A$ and $\B$, the \textit{least upper bound} will be denoted by $\A \vee \B$, namely the finest partition $\C$ for which $\A \preccurlyeq \C$ and $\B \preccurlyeq \C$. Analogously define the \textit{greatest lower bound} of $\A$ and $\B$ by $\A \wedge \B$. 

\begin{example} Consider the two partitions $\A=\{\{1,3,4\},\{2,5\}\}$ and $\B=\{\{1,4\},\{2,3\},\{5\}\}$ of $W=\{1,\dotsc,5\}$ together with a subset
$U=\{1,2,4\}$ of $W$. Then $\A\wedge \B=\{\{1,4\},\{2\},\{3\},\{5\}\}$, $\A\vee\B=\{\{1,\dotsc,5\}\}$, and $\A|^{\pa}_U=\{\{1,4\},\{2\}\}$. 
\end{example}

\paragraph{Möbius functions on the poset of partitions and Möbius inversion.}
The  \textit{Möbius function of a poset} is a  general and powerful tool
in discrete mathematics. It may be considered as a systematic way of 
implementing the inclusion-exclusion principle. We  rely on it
in two contexts in this article: First, we  use it to turn sampling
without replacement into sampling with replacement, and vice versa.
Second, we  need it to turn type frequencies into linkage disequilibra.

Refering to \citet[Prop.~4.6]{Aigner}, let us only summarise here that the M\"obius function $\mu$ is 
defined for all $\A\preccurlyeq \C \in \PP(W)$ via
\begin{equation}\label{help_moeb_Z}
	\sum_{\A \preccurlyeq \underdot{\B} \preccurlyeq \C} \mu (\A, \B) = \begin{cases} 1,\ \A = \C, \\ 0, \text{ otherwise}, \end{cases}
\end{equation}
where the underdot indicates the summation variable.
Let $\A \< \B \in \PP(W)$, with $m=|\B|$ the number of blocks in $\B$, and
$n_j$ the number of blocks of $\A$ within  block $B_j$ of $\B$,
that is, $n_j$ is the number of blocks in $\A|^{\pa}_{B_j}$,
$1 \leqslant j \leqslant m$. The M\"obius function of the pair $(\A,\B)$
is then given by 
\begin{equation}
\mu(\A,\B)
=\prod_{j=1}^m \, \mu \big ( \A |^{\pa}_{B_j}, \1 |^{\pa}_{B_j} \big )
=\prod_{j=1}^m \, (-1)_{}^{n_j-1} (n_j-1)! \,,	\label{moebius_function}
\end{equation}
see \citet[Sect.~7, Ex.~1]{Rota} or \citet[Chap.~3.2, Ex.~4]{Berge}.
We can now state the fundamental M\"obius inversion principle as in \citet[Prop.~4.18]{Aigner}. Let $f$ and $g$
be  mappings from $\PP(W)$ to $\CC$ which are, for all $\A \in \PP(W)$,
related via
\begin{equation}\label{Moebius_sum}
  g(\A) = \sum_{\underdot{\B} \> \A} f(\B).
\end{equation}
Then, this can be solved for $f$ via the inversion formula
\begin{equation}\label{Moebius_inv}
  f(\A) = \sum_{\underdot{\B} \> \A}  \mu(\A,\B) \, g(\B).
\end{equation}
More precisely, this is \emph{inversion from above}.  An analogous
formula applies for \emph{inversion from below}; this relies on refinements
rather than coarsenings, with `$\>$' replaced by `$\<$' in~\eqref{Moebius_sum} and~\eqref{Moebius_inv}.
It is important to note that M\"obius inversion is not restricted to
functions; it also applies to bounded operators.


\section{The model and the genealogical approach}\label{sec:model}
In this section, we define the model formally and motivate our genealogical approach. 
\subsection{The Moran model with single-crossover recombination}
We identify the population at time $t$ by a  (random)  counting measure
$Z_t$ on $\XX$. Namely, $Z_t(\{x\})$ denotes the number of individuals
 of type $x \in \XX$ at time $t$, and
$Z_t(\mathbb{A}) := \sum_{x \in \mathbb{A}} Z_t(\{x\})$ for $\mathbb{A} \subseteq \XX$;
we abbreviate $Z_t(\{x\})$ as $Z_t(x)$. If we define $\delta_x$ as the point measure on $x$ 
(i.e., $\delta_x(y)= \delta_{x,y}$ for $x,y \in \XX$), we can also write
$Z_t = \sum_{x \in X} Z_t(x) \, \delta_x$. Since our Moran population has
constant size $N$, we have $\|Z_t\| = N$ for all times, where 
$\| Z_t \| := \sum_{x \in \XX} Z_t(x)=Z_t(\XX)$
is the norm (or total variation) of $Z_t$.

So, $\{Z_t\}_{t \geqslant 0}$ is  a Markov process in continuous time 
with values in 
\begin{equation}\label{E}
E := \big\{z \in \{0, \dotsc, N\}^{|\XX|} \mid \| z \| = N \big\},
\end{equation}
where $\lvert \XX \rvert$ is the number of elements in $\XX$.
We will describe the action of recombination on (positive) measures with the help of so-called \emph{recombinators} as introduced by \cite{BaakeBaake}; see also \cite{BaakeHerms} for a pedestrian introduction. Let $\mathcal M_+(\XX)$ be the set of all positive, finite measures on $\XX$ and we understand $\mathcal M_+(\XX)$ to include the zero measure. Define the canonical projection  $\pi^{}_I \colon \XX \to \varprod_{i\in I} \XX_i =: \XX_I$, for $I\subseteq S = \{1, \ldots, n\}$, by $\pi^{}_I(x) = (x_i)_{i\in I}$ as usual. 
For $\omega \in \mathcal M_+(\XX)$, the shorthand  $\pi_I^{}.\omega:=\omega \circ \pi_I^{-1}$ indicates the marginal measure with respect to the sites in $I \subseteq S$, where $\pi_I^{-1}$ is the preimage of $\pi_I^{}$. The operation $.$ (where the dot is on the line and should not be
confused with a multiplication sign) is known as the  
\emph{pushforward}
of $\omega$  w.r.t.\ $\pi^{}_I$. In terms of coordinates, 
the definition may be spelled out as
\[ 
  \big(\pi^{}_I . \omega \big) \big(x^{}_I \big) = \omega \circ \pi_I^{-1} \big(x^{}_I \big) 
  = \omega \big(\{x\in \XX \mid  \pi^{}_I(x) = x^{}_I \} \big), \quad 
  x^{}_I \in \XX_I.
\]
Note that $\pi^{}_\varnothing . \omega = \|\omega \| $ and 
$\pi^{}_S . \omega = \omega $. 

Consider now $\A = \{\{1,2, \dotsc, i\},\{i+1, \dotsc, n\}\}
\in \OP_2(S)$ and 
$\omega \in \mathcal M_+(\XX)\setminus 0$, and define
the \emph{projective recombinator} as 
\begin{equation}\label{recombinator_2}
\R{\A}{p}(\omega) := \frac{1}{\|\omega \|^2}\,
\big(\pi^{}_{\{1, \dotsc, i\}} . \omega \big)\otimes 
\big(\pi^{}_{\{i+1, \dotsc, n\}} . \omega\big),
\end{equation}
where $\otimes$ indicates the tensor product (or product measure). Moreover, we set $\R{\tts \1}{p}(\omega): = \omega/\|\omega\|$. $\R{\A}{p}(\omega)$
is a probability measure for \emph{all} $\omega \in \mathcal M_+(\XX)\setminus 0$,
where the zero measure is excluded to make it well-defined.
In words, $\R{\A}{p}$ turns $\omega$ into the (normalised) product measure
of its marginals with respect to the blocks in $\A$. Writing out~\eqref{recombinator_2} in terms of coordinates
gives
\[
\begin{split}
\big( \R{\A}{p} (\omega) \big) (x) & = 
 \frac{1}{\|\omega \|^2}\, \big( \pi^{}_{\{1, \dotsc, i\}} . \omega \big) \big(x^{}_{\{1, \dotsc, i\}} \big)  \, 
\big ( \pi^{}_{\{i+1, \dotsc, n\}} . \omega \big ) \big(x^{}_{\{i+1, \dotsc, n\}}\big)  \\
& = \frac{1}{\|\omega \|^2}\;
\omega (x^{}_1, \dotsc, x^{}_i, \ast, \dotsc, \ast ) \,
\omega (\ast, \dotsc, \ast, x^{}_{i+1}, \dotsc, x^{}_n),
\end{split}
\]
where $\ast$ means marginalisation. 
If $\omega=z$ is the current population, then $\R{\A}{p}(z)$ is the
type distribution  that results when a new individual is created by drawing a leading and (possibly) a trailing segment 
(as encoded by $\A \in \OP_{\leqslant 2}(S)$)
from the current population, uniformly and with replacement. 

\begin{remark}
$\R{\A}{p}$ is a projective version of the recombinator defined by
\cite{BaakeBaake}; it differs from the latter by a factor of $1/\|\omega\|$.
Clearly, both versions agree on the set of probability measures.
As we shall see, the projective version is more suitable in the
stochastic setting, while the original recombinators  are
better adapted to the deterministic situation. 
Since recombinators will only appear in the projective version in this
article, we will drop the superscript and the specification `projective'
and call $\R{\A}{}:= \R{\A}{p}$ a \emph{recombinator} by slight abuse of language.
\end{remark}

In Section~\ref{sec:recombinators}, we will generalise the recombinators and learn more about their probabilistic
meaning and mathematical properties. For the moment, 
let us use them to reformulate the Moran model with recombination in a compact way. Namely, since all individuals die at rate 1, the population loses type-$y$ individuals at rate $Z_t(y)$. Each loss is replaced by a new individual, which is sampled uniformly from $\R{\A}{} (Z_t)$ with
probability $\r{\A}{}$, $\A \in \OP_{\leqslant 2}(S)$. 
Therefore, when $Z_t = z$, the transition to $z + \delta_x - \delta_y$
occurs with rate
\begin{equation}\label{lambda} 
\lambda(z;\, y,\, x) := \sum_{\A \in \OP_{\leqslant 2}(S)} \r{\A}{} \big( \R{\A}{} (z) \big )(x) \,  z(y).
\end{equation}
The summand for $\A=\1$ corresponds to pure resampling, whereas all
other summands include recombination.
Note that $\lambda$   includes `silent transitions' ($x=y$).

\begin{remark}\label{rem:altmodel}
We would like to mention that the model may  alternatively be formulated
in terms of reproducing individuals rather than dying individuals, 
as follows.
Each individual reproduces at rate 1 and picks an $\A \in \OP_{\leqslant 2}(S)$ 
according to the recombination distribution. If $\A \in \OP_2$, the reproducing 
individual contributes the sites in one of the blocks in $\A$ and picks 
a random partner that contributes the sites in the other block to the
offspring.
If $\A = \1$, the reproducing individual contributes all sites.
The offspring pieced together in this way replaces a uniformly
chosen individual from the population. In this formulation, which is closer 
in spirit to the \emph{deterministic} single-crossover model, 
offspring of type $x$ are created
at rate $N \r{\A}{} ( \R{\A}{} (Z_t) )(x)$  and replace an individual of type
$y$ with probability $Z_t(y)/N$. This explains the different 
normalisation of the original recombinator, whereas the additional factor of 
$N=\|Z_t\|$ is absorbed
in its definition in \citet{BaakeBaake}. The resulting 
transition rates, however, are again those in \eqref{lambda}.
Here, we stay with the formulation that led to \eqref{lambda} in the 
first place, since it seems more natural for finite populations. 
\end{remark}

Let us summarise our model as follows:

\begin{definition}[Moran model with single crossovers]\label{def_Moran}
The Moran model with single crossovers is the Markov chain in continuous time
$\{Z_t\}_{t \geqslant 0}$ with state space $E$ of \eqref{E} 
and generator matrix
$\varLambda$ with nondiagonal elements  
\begin{equation*}
\varLambda (z,\, z + w) = \sum_{\substack{x, y \in \XX: \\ \delta_x - \delta_y = w} } \lambda(z;\, y,\, x),\quad w \neq 0,
\end{equation*}
for $z \in E,\ w \in E - z$ (where $E - z := \{v \mid z + v \in E\}$) and $\varLambda (z,\, z) = - \sum\limits_{\substack{v \in E-z:\\ v \neq 0}} \ \varLambda(z, z + v)$.
\end{definition}

\paragraph{Limits of the forward model.}
Consider now the family of processes $\{Z_t^{(N)}\}_{t\geqslant 0}^{}$,
$N=1,2, \dotsc$, where we add the upper index to indicate
dependence on population size. Also consider the normalised version 
$\{Z_t^{(N)}/N\}_{t\geqslant 0}^{}$; $Z_t^{(N)}/N$ is a random probability
measure on $\XX$. For $N \to \infty$ and without any rescaling of the
$\ra$ or of time, the sequence $\{Z_t^{(N)}\}_{t\geqslant 0}^{}$ converges
to the solution of the \emph{deterministic single-crossover equation}
\begin{equation}\label{detreco}
\dot\omega_t=\sum_{\A \in \OP_2(S)} \r{\A}{} \big (\R{\A}{} (\omega_t) - \omega_t \big)
\end{equation}
with initial value $\omega^{}_0$, $\omega^{}_0$ a probability measure, and we assume that 
$\lim_{N \to\infty}  Z_0^{(N)}/N=\omega^{}_0$. This is a 
\emph{dynamical law of large numbers} and due to
\citet[Thm.~11.2.1]{EthierKurtz}. 
The precise statement  as well as the proof are perfectly analogous to Proposition 1 in \citet{BaakeHerms}, 
which assumes a slightly different sampling scheme for recombination. We
therefore leave out the details here.
The deterministic single-crossover equation~\eqref{detreco} was
investigated by \citet{BaakeBaake} and \citet{Baake}. For comparison, note that,
in view of Remark~\ref{rem:altmodel}, 
the probability $\r{\A}{}$ in~\eqref{detreco} is multiplied by the unit rate
at which each individual reproduces, and this way turns into a
recombination \emph{rate}.  

The Moran model with recombination
also has a well-known diffusion limit, which emerges when $N \to
\infty$ under $N \r{\A}{} \to \o{\A}{}$, $\o{\A}{}$ a constant, $\A \in \OP_2(S)$, after 
a speedup of time by a factor of $N$.
In the case of two
loci and two alleles, this goes back to \citet{OhtaKimura}; see also
\citet[Chap.~8.2]{Durrett} for a modern exposition. Two loci with
an arbitrary (but finite) number of alleles are treated by 
\citet{JenkinsFearnheadSong}.
This should readily generalise to the case of a finite number of loci with a
finite number of alleles, but we do not spell it out here, since we will
not draw on the diffusion limit of the forward process later.
\subsection{The ancestral recombination process (ARP) and its marginal version}
In line with standard population-genetic thinking, we employ a genealogical approach by tracing back the ancestry of  (parts of) the genetic material from a population at present that evolved according to the Moran model with single-crossover recombination.  The standard genealogical approach for models with recombination is the  ancestral recombination graph (ARG) first formulated by \citet{Hudson}. 
Today, many different notions of `ARG' are in use. We stick to the usual
convention here that the ARG  assumes the diffusion limit. Hudson's original
version was for two loci, but multilocus generalisations 
\citep{GriffithsMarjoram,BhaskarSong}  
and continuous sequence versions ($n \to \infty$, see, e.g., 
\citealt[Chap.~3.4]{Durrett}) are immediate. 

 \begin{SCfigure}[30][bhtp]
\centering
\fboxsep3mm
{\fbox{{\includegraphics[width=.4\textwidth]{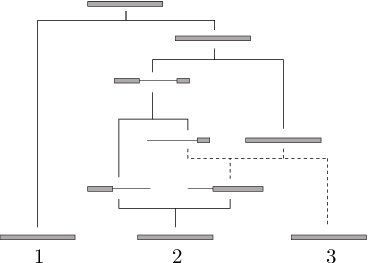}}}}\hfill 
\caption{\label{fig:ARP} A realisation of the full ancestral recombination process, starting from $m=3$ individuals; ancestral material is shaded, non-ancestral material is indicated by thin horizontal lines. The mixed recombination-coalescence event indicated by dashed lines can only appear in the finite population recombination process (ARP). In the diffusion limit, and thus in the ARG, recombination and coalescence act in isolation.} 
\end{SCfigure}

The ARG starts from a sample of individuals from the present population and follows their ancestry backward. When a sequence (or a part of a sequence) experiences a recombination event, it branches into a leading and a trailing segment; when two (parts of) sequences go back to a common ancestor, there is a coalescence event. For overviews see \citet[Chap.~5]{HeinSchierupWiuf}, \citet[Chap.~3.4]{Durrett}, or \citet[Chap.~7.2]{Wakeley}. Mutation can be independently superimposed on the ARG, but will not be considered in this article. One is then interested in the full information on the sample, namely, the probabilities for all possible type distributions of the sample. 
The stationary state of the ARG may be characterised by  a collection of so-called \emph{sampling recursions}; they may be solved analytically for tiny samples (leading to explicit \emph{sampling formulae}), or numerically for larger ones, see \citet{Golding}, or \citet[Chap.~3.6]{Durrett}. But feasibility is limited due to the enormous state space, even for small samples. Alternatively, one resorts to computationally intensive Monte-Carlo or importance-sampling methods to simulate the ARG \citep{GriffithsMarjoram,WangRannala,JenkinsGriffiths}. Recently, Song and coworkers discovered structural properties of the ARG that allow for an efficient combination of analytical and simulation techniques in the regime of \emph{strong recombination} \citep{Song}; more precisely,
they work in terms of expansions in $1/\varrho$ as $\varrho \to \infty$,
where all $\o{\A}{}$, with $\A \in \OP_2(S)$, are assumed to 
scale linearly with the common factor $\varrho$. 

In contrast, we will work in the setting of both \emph{finite} $n$ and \emph{finite} $N$. The corresponding
\emph{ancestral recombination process (ARP)}, which is illustrated in
Figure~\ref{fig:ARP}, is a finite-population version of the multilocus ARG. We then simplify matters by 
only aiming at reduced information. Namely, we consider a partition $\A= \{ A_1,A_2,\dotsc,A_m\}$ of $S$  (with
$m \leqslant N$). Now sample $m$ individuals from the present population and follow back the ancestry of the sites in $A_1$  in the first individual, in $A_2$ in the second individual, $\dotsc$, in $A_m$  in the $m$'th individual, without considering any other sites and any other individuals, as in Figure~\ref{fig:motivation}. That is, each
locus is considered in  one individual only.  The result may be viewed as a \emph{marginalised} version of the ancestral recombination process, and, in the diffusion limit, turns into a marginal version of the multilocus
ARG starting from a sample of size $m$. We will see that this information is sufficient to characterise the time evolution of the expected linkage disequilibria of all orders. We will not employ any scaling or limit, in order to allow for arbitrary strengths of recombination. It will turn out that the approach of \citet{Polen} actually corresponds to this marginal ancestral recombination process, although this is not apparent from their formulation forward in time.

\begin{figure}[t]
\centering
\fboxsep3mm
\subfloat{\fbox{{\includegraphics[width=.45\textwidth]{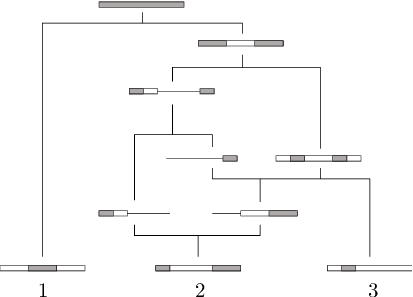}}}}\hfill
\subfloat{\fbox{{\includegraphics[width=.45\textwidth]{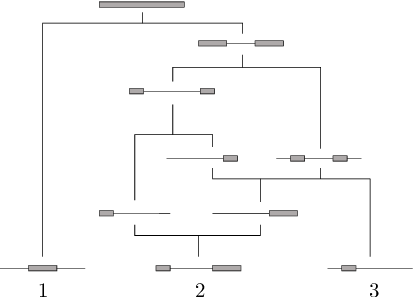}}}}
\caption{\label{fig:motivation} The  marginalised version corresponding to the ARP in Figure~\ref{fig:ARP}, in which we only follow blocks of the partition (shaded), that is, block $A_i$ is sampled in individual number $i$, 
$1\leqslant i\leqslant m$. Material that is ancestral to the sampled individuals, but not to the blocks considered, is shown as open rectangles (left). But since this is not traced back, it can be treated in the same way as material non-ancestral to the sampled individuals (right). Consequently,  the sample will
finally  consist of the blocks of the partition only.}
\end{figure}

More precisely, the letters at the loci considered at present, together with their ancestry, can be constructed by a three-step procedure (see Figure~\ref{backward_process}). First we run a partitioning process $\{\varSigma_t\}_{t\geqslant 0}$ on $\PP(S)$,  backward in time, starting at a given initial partition $\varSigma_0$ with $\lvert \varSigma_0 \rvert = m$. $\varSigma_t$ describes the partitioning of sites into parental individuals at time $t$; sites in the same block (in different blocks) belong to the same (to different)
individuals. Clearly, $\lvert \varSigma_t \rvert$ is
the number of ancestral individuals at time $t$. The process $\{\varSigma_t\}_{t \geqslant 0}$ is independent of the types and will be described in detail in the next section.  In the second step, a letter is assigned to each site of $S$ at time $t$ (i.e.\ in the past) in the following way. For every part of
$\varSigma_t$,  pick an individual from the initial population (without replacement) and copy its letters to the sites in the block considered. For illustration, also assign  a colour to each block, thus indicating different parental individuals. In the last step, the letters and colours are propagated downward (i.e. forward in time) according to the realisation of  $\{\varSigma_t\}_{t\geqslant 0}$ laid down in the first step.
A similar construction was used in the ancestral process by \citet{BaakeWangenheim}, but restricted to a sample of size 1 (i.e. start with $\varSigma_0 = \1$), and in discrete time in the deterministic limit.
Let us now describe the partitioning process in detail.

\begin{figure}[htbp]
\captionsetup[subfigure]{labelformat=empty}
\begin{center}
\subfloat{{\includegraphics[width=.95\textwidth]{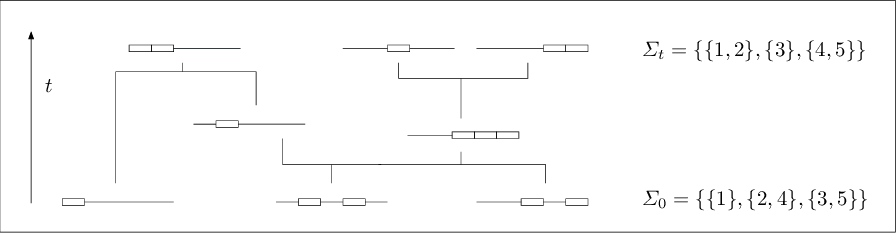}}}\\
\subfloat{{\includegraphics[width=.95\textwidth]{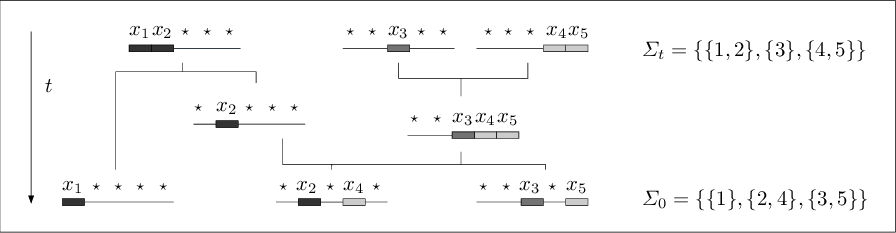}}}
\end{center}
\caption{\label{backward_process} Construction of one possible ancestry of a collection of sites that correspond to the initial partition $\varSigma_0= \{\{1\},\{2,4\},\{3,5\}\}$. The upper panel shows the partitioning process (backward in time). In the lower panel, letters and colours are assigned to each block of $\varSigma_t$  and propagated downward (forward in time). }
\end{figure}


\section{The partitioning process}\label{sec:partitioning}
The partitioning process $\{\varSigma_t\}_{t\geqslant 0}$ is a Markov process on $\PP(S)$, which describes how the sites are partitioned into different individuals backward in time.
Since there is a one-to-one relationship between the individuals and the
blocks of the partition, we may identify individuals with the
ancestral material they carry. 
 
The process $\{\varSigma_t\}_{t\geqslant 0}$ consists of a mixture of splitting (S) 
and coalescence (C) events. It can be constructed independently of the types. In this section, we describe the process by arguing on the grounds of the underlying Moran model; in Section~\ref{sec:duality}, we will formally prove that this is indeed the correct dual process for it.

Since we trace back sites in subsets 
$U \subseteq S$ (rather than complete sequences), we need the corresponding \emph{marginal recombination probabilities}
\begin{equation}\label{def-marginal-rates}
  \r{\, \B}{U} \, := \!
  \sum_{\substack{\A \in \OP_{\leqslant 2}(S) \\ \A |^{\pa}_{U} = \B}} 
  \! \r{\A}{S}
\end{equation}
for any $\B\in\OP_{\leqslant 2}(U)$, where $\r{\A}{S}=\r{\A}{}$.
Note that, for $|U|=1$, the only recombination parameter is $\r{\,\1}{U}=1$. If $U$ is ordered in $S$ (i.e.\ $U = \{x \in S: \min(U) \leqslant x \leqslant \max(U)\}$) and $\B \neq \1 |^{\pa}_U $,
then $\r{\,\B}{U}$ is simply the  probability of crossover after the
(unique) site that leads to partition $\B$. If $U$ is not ordered in $S$,
then $\r{\,\B}{U}$ is the sum of the probabilities of all crossovers that 
lead to partition $\B$, as illustrated in 
Figure~\ref{trapped}. 

Assume now that $U$ is an unordered block of $\varSigma_t$.
This means there is so-called \emph{trapped material}, that is, non-ancestral sites 
enclosed between ancestral regions. All crossover events within
a given trapped segment contribute to the separation of the adjacent
ancestral segments -- in contrast to crossovers in
flanking non-ancestral regions  to the left or the right  of $U$, which
do not affect the genealogy.
Note finally that the upper index in $\r{\,\B}{U}$
can, in principle, be omitted since $U=\cup_{i=1}^{|\B|} \, B_i$, and we will
do so when appropriate. \\

\begin{figure}[htbp]
\centering
\includegraphics[width=.9\textwidth]{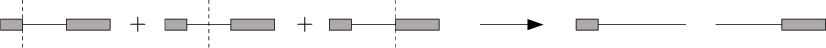}
\caption{\label{trapped}
Let $S=\{1,\dotsc,5\}$ and $U=\{1,4,5\} \subset S$. For the
partition $\B=\{\{1\},\{4,5\}\}$,  there are three  recombination events 
that partition  $U$ into $\B$, thus $\r{\,\B}{U}=r_{\{\{1\},\{2,3,4,5\}\}} + r_{\{\{1,2\},\{3,4,5\}\}} + r_{\{\{1,2,3\},\{4,5\}\}}$.
}
\end{figure}

Now start with the initial partition $\varSigma_0$.
Suppose that the current state is $\varSigma_t = \A=\{A_1,\dotsc,A_m\}$ and denote by $\Delta$ the waiting time to the next event. $\Delta$ is exponentially distributed with parameter $m$, since each block corresponds to an individual, and each individual is independently affected at rate $1$. When the event happens, choose a block uniformly.
If $A_j$ is picked, then $\varSigma_{t+\Delta}$ is obtained as follows (see Figure~\ref{partitioning_process} for an example).\\
In the splitting step, block $A_j$ turns into an intermediate state $\J$
with probability $\r{\J}{\!A_j}$, $\J \in \OP_{\leqslant 2}(A_j) $. In detail: 
\begin{enumerate}[label=(S$_\arabic*$),leftmargin=40pt]
\item \label{split_into_one} With probability $\r{\,\1}{\!A_j}$, the block
$A_j$ remains unchanged. The resulting intermediate state (of this block) is $\J=\1|_{A_j}$. 
\item \label{split_into_two} With probability $\r{\J}{\!A_j}$, $\J\in \OP_2(A_j)$,
 block $A_j$ splits into two parts, $\J=\{A_{j_1},A_{j_2}\}$, which are ordered in $A_j$, but not necessarily in $S$.  
Recall that, via~\eqref{def-marginal-rates}, $\r{\J}{\!A_j}$
takes into account \emph{all} recombination probabilities that lead to
$\J$, including those within trapped material. 
\end{enumerate}
Now, each block of $\J$ chooses out of $N$ parents, uniformly and with replacement. Among these, there are $m-1$ parents that carry one block of $\A_{M \setminus j}=\A \setminus A_j$ each; the remaining $N-(m-1)$ parents are \emph{empty}, that is, they do not carry  ancestral material available for coalescence. Coalescence happens if the choosing block  picks a parent that carries ancestral material; otherwise, the choosing block  becomes an ancestral block of its own, which is available for coalescence from then onwards. The possible outcomes are certain coarsenings of $\A_{M \setminus j}\cup \J$, namely:\\

\noindent If $\J = \{A_j\}$ (case~\ref{split_into_one}), then either 
\begin{enumerate}[label=(C$_{1,\arabic*}$),leftmargin=40pt]
\item \label{coal_none_of_one} With probability $(N - (m - 1))/N$, block $A_j$ does not coalesce with any block of $\A_{M \setminus j}$. As a result, $\varSigma_{t+\Delta}=\varSigma_t=\A$. 
\item \label{coal_one_of_one} With probability $1/N$,  block $A_j$ coalesces with   
block $A_k$,  $k \in M \setminus j$. This results in 
$\varSigma_{t + \Delta}=\A^{}_{M \setminus \{j,k\}} \cup A_{\{j,k\}}$.
\end{enumerate}

\noindent If $\J=\{A_{j_1},A_{j_2}\}$ (case~\ref{split_into_two}), 
we get the following possibilities: 
\begin{enumerate}[label=(C$_{2,\arabic*}$),leftmargin=40pt]
\item \label{coal_none_of_two} With probability $(N-(m-1))(N-m)/N^2$, no block of $\J$ coalesces with a block of $\A_{M \setminus j}$, so 
$\varSigma_{t+\Delta}=\A_{M \setminus j}\cup \J$. 
\item \label{coal_one_of_two} With probability $(N-(m-1))/N^2$, one block of $\J$ coalesces with  block $A_k$, $k \in M \setminus j$, while the other block of $\J$ chooses an empty individual. This ends up in the state $\varSigma_{t+\Delta}= 
\A_{M \setminus \{j,k\}}\cup \{A_{\{j_1,k\}},A_{j_2}\}$ or $\varSigma_{t+\Delta}= \A_{M \setminus \{j,k\}}\cup \{A_{\{j_2,k\}},A_{j_1}\}$. 
That is, in going from $\varSigma_{t}$ to $\varSigma_{t+\Delta}$, either block $A_{j_1}$ or $A_{j_2}$ is moved from $A_j$ to $A_k$. 
\item \label{coal_two_with_each_other} With probability $(N-(m-1))/N^2$, the blocks $A_{j_1}$ and $A_{j_2}$ coalesce with each other, but choose an empty individual, which gives  $\varSigma_{t+\Delta}= \A$.
\item \label{coal_two_of_two} With probability $1/N^2$, the block $A_{j_1}$ coalesces with  $A_k$ and $A_{j_2}$ coalesces  with $A_{\ell}$, $k,\ell \in M \setminus j$. This yields either $\varSigma_{t+\Delta}=
\A_{M \setminus \{j,k,\ell\}} \cup \{A_{\{j_1,k\}}, A_{\{j_2,\ell\}} \}$ if $k \neq \ell$, or $\varSigma_{t+\Delta}= \A^{}_{M \setminus \{j,k\}} \cup A_{\{j,k\}}$ if $k = \ell$.
\end{enumerate}

\begin{figure}[htbp]
\centering \includegraphics[width=.8\textwidth]{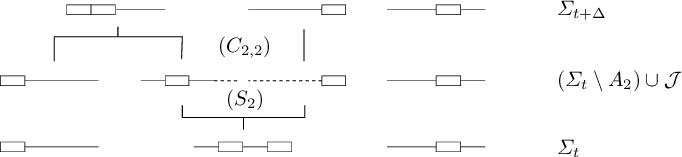}
\caption{One step of the partitioning process with current state 
$\varSigma_t=\{A_1, A_2, A_3 \} = \{\{1\},\{2,4\},\{3\}\}$. In this example, $A_2$ is chosen and splits into $\J=\{\{2\},\{4\}\}$. In the following step~\ref{coal_one_of_two}, the leading part coalesces with $A_1$, whereas the trailing part remains separate, so  that we end up in $\varSigma_{t+\Delta}=\{\{1,2\},\{3\},\{4\}\}$.\label{partitioning_process}}
\end{figure}
Summarising, we see that a transition from $\A$ to $\B$, via partitioning
of block $A_j$ into $\J$, $j \in M$, $\J \in \OP_{\leqslant 2}(A_j)$, 
is possible whenever $\B \succcurlyeq \A_{M \setminus j}\cup \J$ and $\B|^{\pa}_{A_{M \setminus j}}=\A_{M \setminus j},$
or, equivalently, whenever
\[
\B|^{\pa}_{A_j} \> \J \quad \text{and} \quad \B|^{\pa}_{A_{M \setminus j}}=\A_{M \setminus j}.
\]
Each block of
$\J$ coalesces into every block currently available  with probability $1/N$, and remains separate with probability $(N-k)/N$ if there are currently $k$ blocks available; in the latter case, the block considered becomes number $k+1$. We can therefore summarise the rate of the said transition as  
\begin{equation} \label{eq:theta}
\vartheta_{j, \J; \A,\B} = 
\begin{cases}
\r{\J}{\!A_j} \frac{1}{N^{|\J|}}
\frac{ ( N-(m-1)  ) !}{(N - |\B|)!}, & \text{if }
\B|^{\pa}_{A_j} \> \J, \, \B|^{\pa}_{A_{M \setminus j}}=\A_{M \setminus j}, \\[2mm]
0, & \text{otherwise.}
\end{cases}
\end{equation}
Note that this includes silent events where $\B=\A$. 
Thus, the partitioning process $\{\varSigma_t\}_{t\geqslant 0}$ is a continuous-time Markov chain on $\PP(S)$ characterised by the generator 
$\varTheta:=( \varThet{\A \B}{})_{\mathcal A,\mathcal B\in \PP(S)}^{}$
with nondiagonal elements
\begin{equation}\label{Theta}
\begin{split}
 \varThet{\A \B}{} & = \sum_{j\in M} \, \sum_{\J \in \OP_{\leqslant 2}(A_j)} \vartheta_{j, \J; \A,\B}
\\
 & = 
\begin{cases}
	\r{\J}{\!A_j} \frac{1}{N^2} \frac{ ( N-(m-1)  ) !}{(N - |\B|)!}, 
&	\text{if } \B|^{\pa}_{A_j}=\J, \B|^{\pa}_{A_{M \setminus j}}=\A_{M \setminus j},  \text{ for some } j \in M, \, \J \in \OP_2(A_j),   \\[2mm]
	\frac{2}{N^2} + \frac{N-1}{N^2} \big( \r{\,\1}{\! A_j} + \r{\,\1}{\!A_k} \big), 
&	\text{if } \B = \A_{M \setminus \{j,k\}} \cup A_{\{j,k\}} \text{ for some } j\neq k \in M, \\[2mm]
0,
& \text{for all other } \B \neq \A,
\end{cases}
\end{split}
\end{equation}
and  $\varThet{\A \A}{} = - \sum_{\B \in \PP(S) \setminus \A} \varThet{\A \B}{}$. Note that, for $\J \in \OP_2(A_j)$
we have distinguished between $\B |^{\pa}_{A_j} = \J$ and
$\B |^{\pa}_{A_j} = \1 |^{\pa}_{A_j} \succ \J$.  The latter corresponds
to $k=\ell$ in~\ref{coal_two_of_two} and leads to the same transition as  a 
\emph{pure coalescence event} in~\ref{coal_one_of_one}. More precisely,
the total coalescence rate of $j$ and $k$ is
\[
\frac{1}{N} \big(\r{\,\1}{\!A_j} + \r{\,\1}{\!A_k} \big) 
+ \frac{1}{N^2}\, \bigg ( \sum_{\J \in \OP_2(A_j)} \r{\J}{\!A_j} 
+  \sum_{\K \in \OP_2(A_k)}  \r{\tts\K}{A_k} \bigg )
= \frac{2}{N^2} + \frac{N-1}{N^2}\, \big( \r{\,\1}{\!A_j} + \r{\,\1}{\!A_k} \big)   
\]
as stated, since $\sum_{\J \in \OP_2(U)} \r{\J}{U} = 1 - \r{\,\1}{U}$,
$U \subseteq S$. 
Note also that transitions to partitions $\B$ with $|\B|>N$ are impossible,
as it must be.

\begin{remark}
In fact, this generator $\varTheta$ coincides with the generator $\varTheta$ 
worked out by \citet{Polen2} and \citet{Polen} with a very different approach, forward in time.
For $n \leqslant 3$, they state the generator matrices explicitly, and the
identity with~\eqref{Theta} is easily checked by elementwise comparison. 
For $n > 3$, they provide an algorithm,
which runs through all  individuals  and all  sites   
and builds up the matrix $\varTheta$ incrementally, 
in the following manner.
For every given individual, leading and trailing segments (for the split
after site $i$, for all $i \in S \setminus n$)
are taken into account, irrespective of 
whether the segments contain ancestral material. This way,
the algorithm does not distinguish between transitions induced by
recombination events within ancestral (or trapped) material and recombination
events that are invisible in the genealogical perspective, that is,
events that are effectively pure coalescence events. 
Instead,
a  case distinction is performed that is based on whether or not one or both
segments coalesce with individuals that do or do not carry ancestral material.
A detailed investigation of this approach, which involves expanding the cases
into  11 subcases and rearranging these according to the
emerging partitions of the complete ancestral material, leads precisely to
our cases~\ref{coal_none_of_two} to~\ref{coal_two_of_two} (here, both emerging
segments contain ancestral material) and~\ref{coal_none_of_one} and~\ref{coal_one_of_one}
(here one segment is  empty).
Since this approach somehow disguises or mixes the various partitions
of ancestral material that may arise due to a transition, it does
not lead to a closed expression for  $\varTheta$.
In contrast, our approach yields the matrix elements explicitly
for arbitrary $n$, and gives them a natural and plausible meaning in
terms of the partitioning process in backward time. 
\end{remark}

\paragraph{Limits of the partitioning process.}
We now examine how the partitioning process behaves in the two limiting
cases  mentioned in Section~\ref{sec:model}, namely, the deterministic
limit and the diffusion limit. Recall that, in the deterministic limit,
we let $N\to\infty$  without rescaling the recombination probabilities or 
time. Consider, therefore, the family of processes $\{\varSigma_t^{(N)}\}_{t\geqslant 0}^{}$, $N=1,2, \dotsc,$ generated by $\varThet{}{(N)}$, where we again
make the dependence on population size explicit through the upper index.
In the limit, only the pure splitting events~\ref{coal_none_of_two} survive, more
precisely:

\begin{proposition}[Deterministic limit]\label{detlimit}
In the deterministic limit, the sequence of partitioning processes 
$\{\varSigma^{(N)}_t\}_{t\geqslant 0}^{}$ with initial states $\varSigma^{(N)}_0
\equiv \sigma$ converges in distribution to the
process $\{\varSigma_t'\}_{t\geqslant 0}^{}$ with initial state 
$\varSigma'_0 = \sigma$
and generator $\varThetp{}$ defined
by its nondiagonal elements
\[
 \varThetp{\A\B} = \begin{cases} \r{\J}{\!A_j}, & \text{if }  
                  \B = \A^{}_{M  \setminus j} \cup \J \text{ for some }
                   j \in M  \text{ and } \J \in \OP_2(A_j),
                   \\[2mm]
      0, &  \text{for all other } \B \neq \A.
\end{cases} 
\]
Hence, $\{\varSigma_t'\}_{t\geqslant 0}^{}$ is a process
of  progressive refinements, that is,
$\varSigma_\tau'\< \varSigma_t'$ for all $\tau> t$. In particular, if 
$\varSigma_0' \in \OP(S)$,  then $\varSigma_t' \in \OP(S)$ for all times.   
\end{proposition}

\begin{proof}
Inspecting the $N$-dependence of the elements of $\varTheta=\varThet{}{(N)}$
in \eqref{Theta} gives the following order of magnitude
for the nondiagonal elements:
\begin{equation}\label{Theta_ON}
 \varThet{\A \B}{(N)} = 
\begin{cases}
\frac{1}{N^{m + 1 - \lvert \B \vert}}\; \r{\J}{\!A_j} \big (1 + \OO \big(\frac{1}{N}\big)\big), & \text{if }
 \B|^{\pa}_{A_j}=\J, \,  \B|^{\pa}_{A_{M \setminus j}}=\A_{M \setminus j}  \text{ for }
j \in M \,, \J \in \OP_2(A_j), \\[2mm]
 \frac{1}{N} \big( \r{\,\1}{\!A_j} + \r{\,\1}{\!A_k} \big) + \OO\big(\frac{1}{N^2}\big), 
& \text{if } \B = \A_{M \setminus \{j,k\}} \cup A_{\{j,k\}} \text{ for some }
j\neq k \in M, \\[2mm]
0,& \text{for all other } \B \neq \A.
\end{cases}
\end{equation}

Obviously, $\varThet{}{(N)} = \varThetp{} + \OO(1/N)$, which proves convergence of
the sequence of generators of $\{\varSigma^{(N)}_t\}_{t \geqslant 0}^{}$ to that of $\{\varSigma'_t\}_{t \geqslant 0}^{}$. This entails convergence of the
corresponding sequence of semigroups.
With the help of Theorems~4.2.11 and 4.9.10 of \citet{EthierKurtz}, this guarantees
convergence of $\{\varSigma^{(N)}_t\}_{t \geqslant 0}^{}$ to
$\{\varSigma'_t\}_{t \geqslant 0}^{}$ in distribution.

The remainder of the statement is obvious since, under $\varThetp{}$,  the only 
transitions are those that involve the
refinement of a single block, say $A_j$,  into two blocks 
ordered in $A_j$. If $\varSigma'_0$ is ordered in $S$, then all its blocks
are ordered in $S$, and all blocks of $\varSigma'_t$ will be ordered 
in $S$ for all times.
\end{proof}

\begin{remark}
Obviously, in the limit, ancestral material that has been separate will  never come together again
in one individual, such that there are no coalescence events. When starting with $\varSigma'_0=\{S\}$, the genealogy may be represented by a binary tree, which successively branches into smaller  segments; for other initial conditions, one gets a corresponding collection (i.e.\ a forest) of binary trees. We call these trees \emph{ancestral recombination trees} or \emph{ART}s;
a discrete-time analogue was studied by \citet{BaakeWangenheim}.
\end{remark}

We now turn to the diffusion limit and use the factor $N$ rather than (the more common) $2N$ since our $N$ is the \emph{haploid} population size. 
Here, one considers a sequence of processes
in which time is sped up by a factor of 
$N$ and the recombination 
probabilities $\r{\A}{}$ are rescaled such that 
$\lim_{N \to \infty} N \r{\A}{} \to \o{\A}{}$, $\o{\A}{}$ a constant,
for  $\A \in \OP_2(S)$; consequently, 
$\r{\,\1}{} \to 1$
as $N \to \infty$. Note that the $\o{\A}{}$ are \emph{rates} rather
than probabilities. 
The corresponding ARG is the
obvious generalisation of Hudson's original  ARG to $n$  loci,
which we formulate here in our framework for the sake of completeness, as
follows.  Every ordered pair of lines
coalesces at rate $1$; every line splits into two at rate  
$\o{\A}{}$ for every $\A \in \OP_2(S)$, and the  ancestral material 
is distributed between the new lines according to $\A$. 

In this formulation, however, certain silent events are included, 
namely those events that happen in non-ancestral material
flanking the ancestral parts. These events do not affect the partitioning
of ancestral material and  may be removed by working with
the marginalised recombination rates instead.  That is, 
if a sequence currently carries a set $U$ of ancestral sites, then the
relevant recombination rates (in the diffusion limit) are 
$\o{\ts\B}{U}$, with $\B \in \OP_2(U)$, which are defined as 
in~\eqref{def-marginal-rates} but with $r$ replaced by $\varrho$. 
Analogous modifications where the recombination rates depend on the (continuous) region spanned by ancestral material have been investigated by \citet{WiufHein} as well as \citet{McVeanCardin}.

If we now restrict attention to the ancestry
of $n$ loci partitioned between $m$ individuals, we obtain
the \emph{marginal version} of the ARG, which may be formulated as follows.

\begin{definition}[Marginalised $n$-locus ARG] \label{margARG}
Start with the set of $n$ sites distributed across $m \leqslant n$ individuals
(or lines) according to a partition $\varSigma_0''$ with $m$ parts. Throughout
the process, every line is identified with the ancestral material it carries. If it currently
carries ancestral sites $U \subseteq S$, it splits into $\J \in \OP_2(U)$ at rate  $\o{\J}{U}$. 
Every ordered pair of lines coalesces at rate $1$, and so do the ancestral sites they
carry.
That is, the marginalised ARG is the partitioning process $\{\varSigma''_t\}_{t \geqslant 0}^{}$ defined by the
generator $\varThetpp{}$ with nondiagonal elements
\[
\varThetpp{\A \B}=\begin{cases}
\o{\J}{A_j}, & \text{if } \B=\A_{M \setminus j} \cup \J 
\text{ for some } j \in M, \J \in \OP_2(A_j), \\[2mm]
2, & \text{if } \B \succ \A \text{ and } |\B| = |\A| - 1, \\[2mm]
0, & \text{for all other } \B \neq \A.
\end{cases}
\]
\end{definition}

\begin{proposition}[Diffusion limit of the partitioning process]
In the diffusion limit,  the sequence of partitioning processes 
$\{\varSigma^{(N)}_{Nt}\}_{t\geqslant 0}^{}$ with initial states $\varSigma^{(N)}_0
\equiv \sigma$ converges in distribution to the
process $\{\varSigma_t''\}_{t\geqslant 0}^{}$ with initial state 
$\varSigma''_0 = \sigma$ and generator $\varThetpp{}$.
\end{proposition}

\begin{proof}
Due to the rescaling of time, the generator of 
$\{\varSigma^{(N)}_{Nt}\}_{t\geqslant 0}^{}$ has nondiagonal elements
$N \varThet{\A \B}{(N)}$. Referring back to~\eqref{Theta_ON}, they converge
to $\lim_{N \to \infty} N \varThet{\A \B}{(N)} = \varThetpp{\A \B}$,
since we have $\r{\,\1}{U} \to 1$ and $N \r{\J}{U} \to \o{\J}{U}$ for $\J \in \OP_2(U)$.
With the same argument as in the proof of Proposition~\ref{detlimit},
one obtains convergence in distribution as claimed.
\end{proof}

\begin{remark}
As was to be expected, only pure splitting events and pure coalescence
events survive in the diffusion limit. The `mixed transitions', which
involve both splitting and coalescence (i.e.\ the dashed lines
in Figure~\ref{fig:ARP})  vanish under the rescaling;
see also \citet[Fig.~5.11]{HeinSchierupWiuf}. Let us note that several other variants of the recombination process lead to the same diffusion limit. For example, this is true of the simpler (but biologically less realistic) versions of the continous-time Moran model with recombination where recombination is a parallel process that happens independently of reproduction (rather than coupled to reproduction as assumed here), see \citet{BaakeHerms}. Even the discrete-time Wright-Fisher model with recombination lies in the domain of attraction of the diffusion limit.
\end{remark}


\section{Recombinators and sampling functions}
\label{sec:recombinators}

In this section, we will have a closer look at three operators associated with recombination and how they are related to each other. We start by generalising our  recombinators, then introduce closely related sampling functions and finally multilocus correlation functions, known as linkage disequilibria.

\paragraph{Recombinators.} We have already met $\R{\A}{}$ for $\A \in \OP_{\leqslant 2}(S)$;
we now need the generalisation to arbitrary $\A \in \PP(S)$. For 
$\omega \in \M_+(\XX)\setminus 0$, we first define the \emph{non-normalised recombinator} via
\begin{equation}\label{recombinator_general}
\Rbar{\A}{} (\omega) = \big(\pi^{}_{A_1}\!.\omega \big)\,\otimes \cdots \otimes\, \big(\pi^{}_{A_m}\!.\omega \big),
\end{equation}
where it is implied that the product measure refers to the ordering of the sites as specified by the set $S$. In words, $\Rbar{\A}{}$ turns $\omega$ into the product of its marginals with respect
to the blocks in $\A$. 
We will throughout denote non-normalised mappings by an overbar. 
Clearly,
$\Rbar{\tts \varnothing}{}(\omega) = \| \omega \| $, $\Rbar{\tts \1}{}(\omega) = \omega$ and
$  \Vert \Rbar{\A}{}(\omega)  \Vert = \| \omega \|^{\lvert \A \rvert}$.
The corresponding normalised version is 
\begin{equation}\label{recombinator_normalised}
 \R{\A}{}(\omega): = \frac{\Rbar{\A}{}(\omega)}{\big \| \Rbar{\A}{}(\omega) \big \|}\,,  
\end{equation}
which is well-defined since $\omega \neq 0$.
Obviously, $\R{\A}{}(\omega) = \Rbar{\A}{}(\omega /\| \omega \|)$ and $\R{\A}{}(\omega)$
is a probability measure on $\XX$, which
coincides with~\eqref{recombinator_2} for $\A \in \OP_{\leqslant 2}(S)$.

Let us now give a probabilistic interpretation for the case that
a recombinator $\R{\A}{}$ acts on a certain population described by a counting measure 
$z \in E$. 
For the moment, attach labels from the collection $\mathbb{L} := \{1, 2, \dotsc, N\}$ to
the $N$ individuals in the population in a one-to-one manner, and let these individuals 
have (random) types
$X^1_t,X^2_t,\dotsc,X^N_t \in \XX$ at time $t$. The type distribution then is
$Z_t = \sum_{k=1}^N \delta_{X^k_t}$. For $U \subseteq S$ and 
$k \in \mathbb{L}$, let 
$X^k_{t,U} := \pi^{}_U (X^k_t)$,
and consider the following procedure.
Let a partition
$\A = \{A_1, \dotsc, A_m\}$ of $S$ together with a collection of labels 
$\ell = (\ell_1, \dotsc, \ell_m) \in \mathbb{L}^m$ associated with
the blocks be given, i.e. $(\A, \ell) := \{(A_1,\ell_1),  \dotsc, (A_m, \ell_m)\}$. 
Then, piece
together a sequence by taking the sites in $A_1$ from individual $\ell_1$, 
the sites in $A_2$ from individual $\ell_2$, $\dotsc$, and the sites in 
$A_m$ from individual $\ell_m$. The resulting sequence is
$X^\ell_{t,\A} := (X^{\ell_1}_{t,A_1}, \dotsc, X^{\ell_m}_{t,A_m})$. \\
Now, let $\pazocal{L} \in \mathbb{L}^m$ be an $m$-fold random drawing \emph{with replacement} from $\mathbb{L}$, i.e. $\pazocal{L}_1, \ldots, \pazocal{L}_m$ are independent and identically distributed uniform random variables with support $\mathbb{L}$. 
We are now interested in the (random) sequence
\[
X_{t,\A}^{} := X_{t,\A}^{\pazocal{L}}
\]
and the corresponding counting measure 
\[
 \lvert \{ X_{t,\A}^{} = x \}  \rvert = |\{ \ell \in \mathbb{L}^m : X_{t,\A}^\ell = x \}|. 
 \]
This  counts how often one obtains sequence $x$
when performing the above procedure on a population $Z_t$ and all
combinations of individuals are included. 
Since we draw the labels $\pazocal{L}_j$ with replacement (and therefore independently), we can decompose the counting measure 
\begin{equation}\label{sampling_with} 
\big\lvert \big\{ X_{t,\A} = x  \big\} \big\rvert =
\prod_{j \in M} \big\lvert \big\{  \ell_j \in \mathbb{L} : X_{t,A_j}^{\ell_j} = x_{A_j}^{} \big\} \big\rvert
= \prod_{j \in M} \big (\pi^{}_{A_j} . Z_t \big ) (x^{}_{A_j}) 
= \big ( \Rbar{\A}{} (Z_t) \big )(x).
\end{equation}
Clearly, $\R{\A}{} (Z_t)$, the corresponding normalised version, is the
type distribution  that results when a sequence is created by 
taking the letters for the blocks in $\A$ from individuals
 drawn uniformly and with replacement from  the population $Z_t$.
So 
\[
\big (\R{\A}{} (z) \big )(x) = \P \big [X_{t,\A} = x \mid Z_t = z \big],
\]
where $\P$ denotes probability. Note that the left-hand side depends
on time only through the value $z$ of $Z_t$. 

\paragraph{Sampling function.}
For $\A \in \PP(S)$ and $\omega \in \M_+(\XX)\setminus 0$,   
we now define our \emph{sampling function} 
\begin{equation}\label{H_Summe_R}
\Hbar{\A}{} (\omega) := \sum_{\underdot \B \> \A} \mu(\A,\B) \; \Rbar{\tts \B}{} (\omega),
\end{equation}
where $\mu$ is the M\"obius function in \eqref{moebius_function}.
$\Hbar{\A}{} (\omega)$ is not a positive measure in general; but it will
turn out as positive for the important case where $\omega \in E$ with
$\| \omega \| \geqslant |\A|$, see Lemma~\ref{HAnorm}. We will 
therefore postpone the normalisation step.
In any case, M\"obius inversion (compare~\eqref{Moebius_sum} and~\eqref{Moebius_inv}) immediately yields
the inverse of~\eqref{H_Summe_R}:
\begin{fact}\label{R_Summe_H}
For every $\A \in \PP(S)$,
\[ \pushQED{\qed}
\Rbar{\A}{}(\omega) =\sum_{\underdot \B \> \A}  \Hbar{\tts \B}{}(\omega) . \qedhere \popQED
\] 
\end{fact}

We can now give $\Hbar{\A}{}$ a meaning by
reconsidering the procedure that led to~\eqref{sampling_with} but, this
time, individuals are not replaced. 
Therefore, let $\lvert \A \rvert = m \leqslant N$ and $\widetilde{\pazocal{L}} \in \mathbb{L}^m$ be an $m$-fold random drawing \emph{without replacement} from $\mathbb{L}$, i.e. $\widetilde{\pazocal{L}}_1, \ldots, \widetilde{\pazocal{L}}_m$ are now \emph{dependent}.
Then we look at 
\begin{equation}\label{sampling_without}
\widetilde X_{t,\A} := X_{t,\A}^{\tilde{\pazocal{L}}}
\end{equation} 
with the corresponding counting measure
$  \lvert \{ \widetilde X_{t,\A} = x  \} \rvert$.
Since individuals are not replaced, 
an  expression for $\lvert \{ \widetilde X_{t,\A} = x  \} \rvert$
analogous to~\eqref{sampling_with} is 
not immediate.
Instead, we resort to an inclusion-exclusion
argument and prove 
\begin{proposition}\label{prop:sampling_without}
For $\A \in \PP(S)$ with $\lvert \A \rvert = m \leqslant N$ and $Z_t \in E$, we have
\[
\big\lvert \big\{ \widetilde X_{t,\A} = x  \big\} \big\rvert 
= \big\lvert \big\{ \ell \in \mathbb{L}^m : X_{t,\A}^\ell = x \text{ and } \ell_i \neq \ell_j \, \forall\, i \neq j \big\} \big\rvert  
= \big (\Hbar{\A}{} (Z_t) \big )(x).
\]
\end{proposition}
\begin{proof}
Fix a given partition $\A\in \PP(S)$ with 
$\lvert \A \rvert = m \leqslant N$.  
For every $\ell \in \{1,2,\dotsc, N\}^m$,
the pair $(\A,\ell)$  uniquely defines a pair $(\B, \tilde \ell)$, 
where $\tilde \ell \in \{  \ell \in \{1,2,\dotsc, N\}^{\lvert \B \rvert} : \ell_i \neq \ell_j \, \forall \,
i \neq j \}$ and $\B \> \A$, as follows. Join all blocks of $\A$
that have the same label, and attach that label to the new block.
The result is $(\B, \tilde \ell)$. The other way round, every 
$(\B, \tilde \ell)$ with $\B \> \A$ and
$\tilde \ell \in \{  \ell \in \{1,2,\dotsc, N\}^{\lvert \B \rvert}: \ell_i \neq \ell_j \, \forall\, i \neq j \}$ uniquely defines the labelling $\ell$
of the blocks of $\A$ (keep in mind that $\A$ is fixed): block 
$A_k \in \A$ receives the
label of that  block $B_j \in \B$ in which it is contained.
We can therefore identify the set 
$\{(\A, \ell) : \ell \in \{1,2,\dotsc, N\}^m\}$ with the set 
$\bigcup_{\underdot \B \> \A} \{(\B, \tilde \ell) : 
\tilde \ell \in \{ \ell \in \{1,2,\dotsc, N\}^{|\B|} : \ell_i \neq \ell_j
\, \forall\,  i \neq j \}\}$ and \emph{decompose} the event $\{  X_{t,\A} = x \} = \dot{\bigcup}_{\underdot \B \> \A} \{ \widetilde X_{t,\B} = x   \}$, 
which entails
\[
 \big\lvert \big\{  X_{t,\A} = x \big \} \big\rvert
  = \sum_{\underdot \B \> \A} \big \lvert \big \{ \widetilde X_{t,\B} = x \big \} \big \rvert.
\]
By~\eqref{sampling_with}, the left-hand side equals $(\Rbar{\A}{} (Z_t))(x)$.
Due to the M\"obius inversion principle (applied backward), 
$\lvert \{ \widetilde X_{t,\B} = x  \} \rvert$ on
the right-hand side must equal $(\Hbar{\tts \B}{} (Z_t) )(x)$, as claimed. 
\end{proof} 
 
\begin{lemma}\label{HAnorm}
For $\A \in \PP(S)$ with $|\A|=m \leqslant N$ and $z \in E$,
$\Hbar{\A}{}(z)$ is a positive measure with
\[
\big \| \Hbar{\A}{} (z) \big \|= N \,(N - 1)\cdots (N - m + 1) > 0.
\]
\end{lemma}
\begin{proof}
Since, under the given assumptions, $(\Hbar{\A}{}(z))(x) 
= \lvert \{ \widetilde X_{t,\A} = x \mid Z_t=z  \} \rvert \geqslant 0$ for
all $x$ by Proposition~\ref{prop:sampling_without}, it is a positive measure,
and its norm can be evaluated via  
\[
\big \| \Hbar{\A}{}(z) \big \| =  
\sum_{x \in \XX} \big \lvert \big \{ \widetilde X_{t,\A} = x \mid Z_t=z \big \}\big  \rvert.
\]
By means of~\eqref{sampling_without}, this equals the number of possibilities
of how to choose $m$ labelled individuals out of $N$ ones \emph{without} 
replacement, where the order
is respected; this is $ N \,(N -1)\cdots (N -m+1)$, which is positive
since $m \leqslant N$. 
\end{proof}

Under the assumptions of Proposition~\ref{prop:sampling_without},
we can therefore define the normalised version of $\Hbar{\A}{}(z)$:
\begin{equation}\label{H_normalised}
\H{\A}{}(z) :=  \frac{\Hbar{\A}{}(z)}{\big \| \Hbar{\A}{}(z) \big \|} =
\frac{(N - m)!}{N!} ~ \Hbar{\A}{}(z).
\end{equation}
$\H{\A}{} (z)$ is the type distribution that results when a sequence is created by taking
the letters for the blocks as encoded by $\A$ from individuals
 drawn uniformly and \emph{without replacement} from  the population $z$,
hence
\[
\big ( \H{\A}{} (z) \big ) (x) = \P \big [\widetilde X_{t,\A} = x \mid Z_t = z \big].
\]
$\H{\A}{}$ will later serve as duality function.
The situation described here is exactly what happens when a sample is
taken in our marginal ancestral recombination process: either the initial sample
(according to $\varSigma_0$, from the present population $Z_t$) or 
the ancestral one (according to $\varSigma_t$, from the initial population
$Z_0$) -- hence
our name \emph{sampling function}.
In this light, Fact~\ref{R_Summe_H} expresses counting with 
replacement in terms of counting  without replacement, provided $\omega$ is a counting measure. 

It is also instructive to express the normalised sampling functions
in terms of the normalised recombinators. For 
$z \in E$ and $|\A| \leqslant N$, this gives, via~\eqref{recombinator_normalised}, 
\[
\H{\A}{}(z)  = \sum_{\underdot \B \> \A} 
\frac{(N - \lvert \A \rvert)!\, \Ncard{\B}}{N!} \, \mu(\A,\B)\, \R{\tts \B}{} (z).
\]
Note that $(N - \lvert \A \rvert)!\, \Ncard{\B}/N! 
= \OO(\Ncardd{\B}{\A})$. This illustrates how the
inclusion of coarser partitions yields higher-order correction terms.
The other way round, using~\eqref{recombinator_normalised}, 
Fact~\ref{R_Summe_H}, and~\eqref{H_normalised}, one gets  
\begin{equation}\label{RAHB}
\R{\A}{} (z) = \sum_{\underdot \B \> \A} 
\frac{N!}{\Ncard{\A}(N - \lvert \B \rvert)!}\, \H{\tts \B}{} (z).
\end{equation}

\paragraph{Restriction to subsystems.}
Recall that we write the restriction of a measure $\omega \in \M_+(\XX)$ to a subspace $\XX_U:=\bigtim_{i\in U} \XX_i$ of $\XX$  as $\pi^{}_U .\omega := \omega \circ \pi^{-1}_U$, which corresponds to  marginalisation. 
Clearly we can also define recombinators for any non-empty
subset $U \subseteq S$ and any partition $\A = \{A_1, \dotsc, A_m\} \in \PP(U)$ as 
$\Rbar{\A}{U} (\pi^{}_U .\omega)$, in perfect analogy with 
$\Rbar{\A}{S} (\omega)$ for $\A \in \PP(S)$, which is 
$\Rbar{\A}{} (\omega)$; and likewise for $\R{\A}{U}$, $\Hbar{\A}{U}$, and
$\H{\A}{U}$ (if $\omega \neq 0$). For clarity, we sometimes denote the subsystem by a superscript. 
However, as in the case of the marginal recombination probabilities,
it can be dispensed with since $U=\cup_{j=1}^{|\A|} \, A_j$ if $\A \in \PP(U)$.
The interpretation in terms of sampling,
as well as Fact~\ref{R_Summe_H}, carry over.

Let us collect some basic properties of recombinators:

\begin{fact} \label{properties_R}
For $\A, \B \in \PP(S)$ and $U,V\subseteq S$ with $S=U\, \dot \cup\, V$ one has
\begin{enumerate}[label=(\Alph*)]
\item \label{Rinf} 
$\R{\A}{} \R{\tts \B}{} = \R{\A \wedge \B}{}$. \vspace*{1.8mm}
\item \label{R_on_subsystem}
$\pi^{}_U. \R{\A}{S} (\omega)= \R{\A|^{\pa}_U}{U}(\pi^{}_U.\omega)$.
\item \label{productR}
If in addition $\A\< \{U,V\}$, then $\Rbar{\A}{S} = \Rbar{\A|^{\pa}_U}{U} \otimes \, \Rbar{\A|^{\pa}_V}{V}$. Explicitly, this reads \\[.3em]
$\Rbar{\A}{S}(\omega)= 
\big (\Rbar{\A|^{\pa}_U}{U} \otimes \, \Rbar{\A|^{\pa}_V}{V} \big)(\omega) =
\big (\Rbar{\A|^{\pa}_U}{U}(\pi^{}_U . \omega) \big ) \otimes  
\big (\Rbar{\A|^{\pa}_V}{V} (\pi^{}_V . \omega) \big ). $
\end{enumerate}
\end{fact}
Here and in what follows, we may omit the argument when the meaning is
clear.

\begin{proof}[Proof of Fact~$\ref{properties_R}$]
Property~\ref{Rinf} is Proposition~2 and property~\ref{R_on_subsystem} is
Lemma~1 of \citet{BaakeBaakeSalamat} (they both remain true in our normalisation). Property~\ref{productR} is an obvious generalisation of Proposition~2 of \citet{WangenheimBaakeBaake}. It is easily seen 
by using first property~\ref{Rinf}, then~\eqref{recombinator_general}, then~\ref{R_on_subsystem} and finally~\eqref{recombinator_general} once more to give
\begin{align*}
\Rbar{\A}{S}(\omega) 
& = \Rbar{\tts\{U,V\}}{S} \big(\Rbar{\A}{S}(\omega)\big)
= \big (\big ( \pi^{}_{U} . \Rbar{\A}{S} \big ) \otimes \big ( \pi^{}_{V} . \Rbar{\A}{S} \big) \big ) (\omega) \\
& = \big  ( \Rbar{\A| ^{\pa}_U}{U} (\pi^{}_{U} . \omega) \big ) \otimes \big ( \Rbar{\A |^{\pa}_V}{V} (\pi^{}_{V} . \omega) \big ) = \big(\Rbar{\A|^{\pa}_U}{U} \otimes \, \Rbar{\A|^{\pa}_V}{V} \big) (\omega). \qedhere
\end{align*}
\end{proof}

Let us note a connection between recombination and sampling that will be important in what follows.

\begin{lemma} \label{product_H_R}
Let $S = U\, \dot \cup \, V$ for two nonempty subsets $U, V \subseteq S$. For two partitions $\A \in \PP(U)$, $\B \in \PP(V)$, the recombinator and the sampling operator satisfy
\[
\Rbar{\A}{U} \otimes \Hbar{\tts \B}{V} = \sumsubstack{\underdot \C\> \A \cup \B}{\C|^{\pa}_V = \B} \Hbar{\tts \C}{S}.
\]
\end{lemma}

\begin{proof}
Using~\eqref{H_Summe_R} followed by Fact~\ref{properties_R}~\ref{productR} and 
Fact~\ref{R_Summe_H} we get
\[
\Rbar{\A}{U} \otimes \Hbar{\tts \B}{V} = \Rbar{\A}{U} \otimes \Big (\sum_{\underdot \D \> \B} \mu(\B, \D)\, \Rbar{\tts \D}{V} \Big ) = \sum_{\underdot \D \> \B} \mu(\B, \D)\, \Rbar{\tts \D\cup \A}{S} = \sum_{\underdot \D \> \B} \mu(\B, \D) \sum_{\underdot \Ee \> \D \cup \A} \Hbar{\tts \Ee}{S}.
\]
Changing the summation order and applying~\eqref{help_moeb_Z}
finally leads to
\[
\Rbar{\A}{U} \otimes \Hbar{\tts \B}{V} = \sum_{\underdot \C \> \A \cup \B} \Hbar{\tts \C}{S} \sum_{\B \< \underdot \D \< \C|^{\pa}_{V}} \mu(\B, \D) = 
\sumsubstack{\underdot \C \> \A \cup \B}{\C|^{\pa}_{V} = \B} \Hbar{\tts \C}{S}. \qedhere
\]
\end{proof}

\begin{remark}
In a perfectly analogous way, one can show
\[
\Hbar{\A}{U} \otimes \Hbar{\tts \B}{V} = \sumsubstack{\underdot \C\> \A \cup \B}{\C|^{\pa}_U = \A,\, \C|^{\pa}_V = \B} \Hbar{\tts \C}{S}.
\]
This illustrates once more that, unlike the $\Rbar{\A}{}$, the $\Hbar{\A}{}$ do \emph{not} have a product
structure; this  reflects the dependence inherent to drawing without replacement. 
\end{remark}


\paragraph{Correlations (or linkage disequilibria).}
Linkage disequilibria (LDE) are used in population genetics to quantify the deviation from independence
of allele frequencies at the various sites in a sequence. From three sites
onwards, many different notions of linkage disequilibria are available in the literature,
see \citet[Chap. V.4.2]{Buerger} for an overview. 

We will use as LDEs the general correlation functions, which are widely used in statistical physics, see \citet{Dyson} or \citet[Chap. 5.1.1]{Mehta}. This results in an explicit formula for multilocus LDEs for an arbitrary number of sites in terms of sums of products of marginal frequencies, see also \citet[Appendix]{BaakeBaake} or \citet{GorelickLaubichler}. As we will see, common definitions for two and three sites coincide with ours. \\
For any given subset $U \subseteq S$ and $\A \in \PP(U)$, we first define
\emph{correlation operators} as
\begin{equation}
\L{\A}{U} = \sum_{\underdot \B\< \A} \mu(\B,\A)\, \R{\tts \B}{U}. \label{Correlation_function}
\end{equation}
Note that the summation is now over all \emph{refinements} of $\A$, in contrast to our sampling functions, which involve all coarsenings of $\A$. The restriction to subsystems stems from the fact that one usually considers deviation from independence on  small subsets of $S$.

The $\L{\A}{U}$ have a product structure,
$
 \L{\A}{U} = \prod_{j=1}^{|\A|}\, \L{\tts \1}{A_j}\,,
$
which is obvious from~\eqref{Correlation_function} together with the product structure
of the recombinators (Fact~\ref{properties_R}~\ref{productR}) and that
of the Möbius function \eqref{moebius_function}.
Eq.~\eqref{Correlation_function} has the inverse
\begin{equation*}
\R{\A}{U} = \sum_{\underdot \B \< \A}  \L{\tts\B}{U} = \sum_{\underdot \B \< \A}\; \prod_{j=1}^{\lvert \B \rvert}\, \L{\tts\1}{B_j}\,
\end{equation*}
due to \emph{inversion from below} (see Section~\ref{sec:prelim}).
The latter can be reformulated as
\begin{equation}\label{recursive_LDE}
\L{\A}{U} = \R{\A}{U} - \sum_{\underdot \B\prec \A} \;
\prod_{j = 1}^{\lvert \B \rvert}\, \L{\,\1}{B_j}.
\end{equation}

The case $\A=\1|_U$, $U\subseteq S$, now is of special interest.
In line with population-genetics understanding, we define the \emph{multilocus linkage disequilibrium with respect to the sites in $U$} by letting 
$\L{\tts \1}{U}$ act on the marginal measure $\pi^{}_U . \omega$, $\omega \in \mathcal M_+(\XX) \setminus 0$:
\[
\L{\tts \1}{U}(\pi^{}_U .\omega)=\sum_{\A\in \PP(U)}  \mu(\A,\1 |^{\pa}_{U})\, \R{\A}{U} (\pi^{}_U .\omega), 
\]
cf.~\eqref{Correlation_function}.
Note that $\L{\1}{U}(\pi^{}_U . \omega)$ is again a measure on $\pi^{}_U(\XX)$, but no longer positive in general.
With the help of~\eqref{recursive_LDE}, it can be reformulated as
\[
\L{\tts \1}{U} (\pi^{}_U . \omega) = \R{\tts \1}{U} (\pi^{}_U . \omega)
-\sum_{\underdot \B \prec \1|_U} \; \prod_{j = 1}^{\lvert \B \rvert}\, \L{\tts \1}{B_j} (\pi^{}_{B_j} . \omega),
\]
which is Eq.~(1) in \citet{GorelickLaubichler}. Likewise, this alternative formulation of multilocus LDEs agrees with previous ones from \citet{Geiringer}, \citet{Bennett},
and \citet{Hastings} up to $|U| \leqslant 3$.

\begin{example} 
 For $S = \{1,2,3,4\}$ the LDE with respect to the sites in $U = \{2,4\}$ reads
\begin{align*}
\L{\tts \1}{U} (\pi_{\{2,4\}}^{} . \omega)(x)
& = \R{\tts \1}{U} (\pi_{\{2,4\}}^{} . \omega)(x) - \R{\tts \{\{2\},\{4\}\}}{U} (\pi_{\{2,4\}}^{} . \omega) (x) \\[1mm]
&  = \frac{1}{\| \omega \|}\, \omega(\ast,x_2,\ast,x_4) - \frac{1}{\| \omega \|^2}\, \omega(\ast,x_2,\ast,\ast)\, \omega(\ast,\ast,\ast,x_4).
\end{align*}
Similarly for $U=\{1,3,4\}$ we get
\begin{align*}
\L{\tts \1}{U} (\pi_{\{1,3,4\}}^{} . \omega)(x) & =  \frac{1}{\| \omega \|}\, \omega(x_1,\ast,x_3,x_4) -\frac{1}{\| \omega \|^2}\, \omega(x_1,\ast,\ast,\ast)\, \omega(\ast,\ast,x_3,x_4)\\ 
& \quad -\frac{1}{\| \omega \|^2}\,\omega(x_1,\ast,x_3,\ast)\, \omega(\ast,\ast,\ast,x_4) -\frac{1}{\| \omega \|^2}\, \omega(x_1,\ast,\ast,x_4)\, \omega(\ast,\ast,x_3,\ast)\\
& \quad + 2\,\frac{1}{\| \omega \|^3}\, \omega(x_1,\ast,\ast,\ast)\, \omega(\ast,\ast,x_3,\ast)\, \omega(\ast,\ast,\ast,x_4). 
\end{align*}
\end{example}

The correlation operators can also be expressed in terms of our sampling operators. Eqns.~\eqref{Correlation_function} and~\eqref{RAHB}, together with
a change of the summation order, lead to
\begin{equation} \label{L_in_H}
\L{\A}{U}=\sum_{\underdot \B \< \A} \mu(\B,\A) \sum_{\underdot \C \> \B} \frac{N!}{(N - \lvert \C \rvert)!\,\Ncard{\B}}\, \H{\tts \C}{U} = \sum_{\C\in \PP(U)} \H{\tts \C}{U} \sum_{\underdot \B\<\A \wedge \C} \frac{N!}{(N - \lvert \C \rvert)!\,\Ncard{\B}}\, \mu(\B,\A).
\end{equation}
For a counting measure $z \in E$ and $U \subseteq S$ with $|U| = k  \leqslant 3\leqslant N$, Eq.~\eqref{L_in_H} yields a particularly nice explicit expression for the LDEs:
\begin{equation}\label{LDEexplicit}
\L{\tts \1}{U} (\pi^{}_U . z) = \frac{N!}{N^k (N-k)!} \sum_{\A \in \PP(U)} \mu(\A,\1 |^{\pa}_{U})
\, \H{\A}{U} (\pi^{}_U . z),
\end{equation}
as is easily verified.
For larger $k$, the explicit formula gets more involved.

Let us now consider $\L{\A}{U}$ for $\A \in \PP(U) \setminus \! \1 |^{\pa}_{U}$. Due to
its product structure, the collection of all $\L{\tts \1}{V}(\pi^{}_V . \omega)$,
$V \subseteq U$, determines all $\L{\A}{U}(\pi^{}_U . \omega)$, $\A \in \PP(U)$.
This is why, for a deterministic $\omega$, the $\L{\A}{U}(\pi^{}_U . \omega)$, $\A \neq \1 |^{\pa}_{U}$,
are of no particular interest of their own.
This changes, however, when
$\omega$ is  random  (like $Z_t$). For we typically do
not know  the law of $Z_t$ completely; rather, we have access to
the expectation of  certain functions of $Z_t$.
More precisely, let $\varphi$ be the law of $Z_t$
and $\E_\varphi$ denote the expectation with respect to  $\varphi$
(that is, for a function $f$ of $Z_t$,
$\E_\varphi[f] = \int f(z) \, {\rm d} \varphi(z)$). It is important to
note that  the
product structure of the recombined measure does not carry over to
the expectation. That is, for $\A \in \PP(U)$,
$\E_\varphi [\R{\A}{U} (\pi_U^{} . Z_t)] \neq \R{\A}{U} (\E_\varphi[\pi_U^{} . Z_t])$ in general,
see the discussion in \citet{BaakeHerms}; 
this is indeed a subtle point that  sometimes goes wrong, as in
\citet{Polen4}, Eq.~(12), or \citet{Polen}, pp.~471/472.
As a consequence, one also has
$\E_\varphi[\L{\A}{U}(\pi^{}_U . Z_t)]
\neq \prod_{i=1}^{|\A|} \L{\tts \1}{A_i} (\E_\varphi[ \pi^{}_{A_i} . Z_t])$ in general.
In the stochastic case, therefore, it is interesting to consider
the $\L{\A}{U}$ for $\A \neq \1 |^{\pa}_{U}$ as well.
The expectations $\E_\varphi [\L{\A}{U}(\pi^{}_U . Z_t)]$
contain information on how the mean LDEs in one part
of the sequence depend on the mean LDEs in other parts of the
sequence. In the next section,
we will obtain an ODE system for the $\E_\varphi[ \H{\A}{S} (Z_t)]$, $\A \in \PP(S)$,
and these translate into $\E_\varphi[\R{\A}{S} (Z_t)]$ and thus into
$\E_\varphi[\L{\A}{S}(Z_t)]$ via~\eqref{L_in_H}. Marginalisation can then
be used to calculate the corresponding quantities on $U \subset S$,
such as  $\E_\varphi  [\L{\A}{U}(\pi^{}_U . Z_t)]$ for $\A \in \PP(U)$. 
\section{Duality}\label{sec:duality}
Duality is a powerful tool to obtain information about one process by studying 
another, the dual process. The latter may, in an optimal case, have a much smaller state space than the original one. Duality results are essential in interacting particle systems in physics and in population genetics. They are often related to time reversal. The most famous example in population genetics is arguably the moment duality between the Wright-Fisher diffusion forward in time and the block counting process of Kingman's coalescent backward in time \citep{Donnelly, Moehle}. \citet{Mano} extended this result by incorporating recombination into the two-locus, two-allele case. His results are based on the original version of the ARG and thus on the diffusion limit.

We will briefly explain the general duality concept and then prove that 
our processes $\{Z_t\}_{t \geqslant 0}$ and $\{\varSigma_t\}_{t \geqslant 0}$
are duals of each other.
For the general principle, let $X=\{X_t\}_{t\geqslant 0}$ and $Y=\{Y_t\}_{t\geqslant 0}$ be two Markov processes with state spaces $E$ and $F$. Define by $M(E\times F)_b$ the set of all bounded measurable functions on $E\times F$. The following definition of duality with respect to a function goes back to \citet{Liggett}; see also the recent review by \citet{KurtJansen}.

\begin{definition}[Duality]
The Markov processes $X$ and $Y$, with laws $\varphi$ and $\psi$, 
respectively, are said to be \textit{dual} with respect to a function
$H\in M (E\times F)_b$ if, for all $x\in E$, $y\in F$ and $t\geqslant 0$,
\begin{equation}\label{duality}
\E_\varphi\left[H(X_t,y) \mid X_0=x\right]=\E_\psi\left[H(x,Y_t) \mid Y_0=y\right]. 
\end{equation}
\end{definition}

If $E$ and $F$ are  finite, every function $H\in M(E\times F)_b$ may be represented by a  matrix with bounded entries $H(v,w)$, $v \in E$, $w \in F$. If, further, $X$ and $Y$ are  time-homogeneous with generator matrices $\varLambda$ and $\varTheta$ respectively, the expectations in \eqref{duality}
may be written in terms of the corresponding semigroups, i.e.,
\begin{equation}
\begin{split}\label{semigroup}
\E_\varphi \left[H(X_t,y) \mid X_0 = x \right] & = \sum_{v \in E} ({\rm e}^{t \varLambda})_{xv} H(v,y),\\ 
\E_\psi\left[H(x,Y_t) \mid Y_0 = y \right] & = \sum_{w \in F} ({\rm e}^{t \varTheta})_{yw} H(x,w). 
\end{split}
\end{equation}
Since the duality equation~\eqref{duality} is automatically satisfied at $t=0$,
it is  sufficient to check the
identity of the derivatives at $t=0$. That is, 
Eq.~\eqref{duality}  holds for all times if and only if
\begin{equation}\label{finite_duality}
\begin{split}
 \frac{{\rm d}}{{\rm d}t}\, \E_\varphi \left[H(X_t,y) \mid X_0 = x \right ] { \big |_{t=0}}  & = 
 \sum_{v\in E} \varLambda_{xv}\, H(v,y) \\
 & =  
  \sum_{w\in F} H(x,w)\,\varTheta_{yw}
= \frac{{\rm d}}{{\rm d}t}\ \E_\psi \left[ H(x,Y_t) \mid Y_0 = y \right ] { \big |_{t=0}}
\end{split}
\end{equation}
for all $x\in E,\,y\in F$. As a short-hand of \eqref{finite_duality}, one can write $\varLambda\, H=H \varTheta^T$, where $T$ denotes transpose.

We will now present a duality result that  justifies our construction 
of a marginalised sample at present via the partitioning process and
sampling from the initial population (cf. Figure~\ref{backward_process}). 
Indeed, it is not coincidence that we have denoted our sampling functions by $\H{\A}{}$ and our 
generators by $\varLambda$ and $\varTheta$.

\begin{theorem} \label{equality_generators}
The population process $\{Z_t\}_{t\geqslant 0}$ and the partitioning process $\{\varSigma_t\}_{t\geqslant 0}$
with the generators $\varLambda$ and $\varTheta$ and resulting laws
$\varphi$ and  $\psi$, respectively,
are dual with respect to the sampling function $H$ defined in \eqref{H_normalised}. Explicitly,
\begin{equation}\label{duality_Z_Sigma}
 \E_{\varphi} \big[ \H{\A}{} (Z_t) \mid  Z_0 = z \big ] = 
 \E_{\psi} \big[ \H{\varSigma_t}{}(z) \mid \varSigma_0=\A \big]
\end{equation}
for all $\A \in \PP(S)$ and $z \in E$.
\end{theorem}

Before we embark on the proof, let us briefly comment on the meaning of
this result. 
\begin{remark} 
Eq.~\eqref{duality_Z_Sigma} is the formal equivalent of the construction
in Figure~\ref{backward_process}. To see this, recall the
random variables $\widetilde X_{t,\A}$ from~\eqref{sampling_without}.
With their help, the left-hand side
of~\eqref{duality_Z_Sigma} may be reformulated as a probability
distribution,
\[
\E_{\varphi} \big [ \H{\A}{} (Z_t) \mid  Z_0 = z \big ] = 
\E_{\varphi}  \big [\P \big [ \widetilde X_{t,\A} = \cdot \big] \mid  Z_t,  Z_0 = z  \big ]
= \P_\varphi \big [\widetilde X_{t,\A} = \cdot \mid   Z_0 = z  \big ],
\] 
since 
the expectation  is over all realisations of 
$Z_t$. The  right-hand side is the probability distribution considered
by \citet{Polen}. Likewise, the right-hand side of~\eqref{duality_Z_Sigma} is equal to
\[
\E_{\psi} \big [ \H{\varSigma_t}{} (z) \mid  \varSigma_0 = \A \big ] = 
\E_{\psi} \big [\P \big [ \widetilde X_{0,\varSigma_t} = \cdot \big ] \mid \varSigma_t,  \varSigma_0 = \A  \big ]
= \P_\psi \big [ \widetilde X_{0,\varSigma_t} = \cdot \mid   \varSigma_0=\A \big ],
\]
since the expectation is over all realisations of
$\varSigma_t$. The right-hand side is the distribution of types when
sampling from the initial population according to the partition $\varSigma_t$,
where it is understood that the initial population consists of
the types $X_0^1, \dotsc, X_0^N$ with $\sum_{k=1}^N \delta_{X_0^k}=z$.
Recall that time runs forward in  $Z_t$, $X_t^k$, and $\widetilde X_{t,\A}$,
but backward in $\varSigma_t$.
\end{remark}

In order to avoid case distinctions in the calculations in the 
remainder of this section, let us agree on the following conventions 
concerning (partitions of) empty sets. Namely, we  set
$\A_\varnothing := \varnothing$, 
$\Hbar{\tts \varnothing}{} (\pi^{}_\varnothing . z) 
=\Rbar{\tts \varnothing}{} (\pi^{}_\varnothing . z) := \pi^{}_\varnothing . z 
= \| z \| = N$, and $\mu(\varnothing, \varnothing) := 1$.
We now collect some auxiliary results in the
following Lemma.

\begin{lemma}\label{sumysumx}
Consider a counting measure $z \in E$, a partition $\A \in \PP(S)$  with $|\A| = m \leqslant N$ and corresponding index set 
$M = \{1, \dotsc, m\}$,  and a partition $\B \in \PP(S)$.
Then, the following statements hold:
\begin{enumerate}[label=(\Alph*)]
\item \label{sumy_H_z_plus_deltax}
 $\displaystyle{\sum_{x \in \XX}} \big( \R{\tts \B}{} (z)\big)(x) \,
\big[ \Hbar{\A}{} (z + \delta_x) - \Hbar{\A}{} (z) \big]
 =  \displaystyle{\sum_{j \in M}} 
\Big ( \Hbar{\A_{M \setminus j}}{} \otimes \R{\tts \B|^{\pa}_{A_j}}{} \Big )(z).$  
\item \label{sumy_H_z_minus_deltay}
${\displaystyle \sum_{x \in \XX} z(x)}  \,
\big [ \Hbar{\A}{} (z - \delta_x) - \Hbar{\A}{} (z) \big ]  
 =  - m \, \Hbar{\A}{}(z)$.
\end{enumerate}
\end{lemma}
%
Before we prove the lemma, let us give some explanations.

\begin{remark}
Note first that, with the above conventions, 
the right-hand side of identity~\ref{sumy_H_z_plus_deltax} evaluates to
\[
\sum_{j \in M} 
\Big ( \Hbar{\A_{M \setminus j}}{} \otimes \R{\tts \B|^{\pa}_{A_j}}{} \Big )(z)
= N \R{\tts \B}{} (z) \quad \text{if } \A = \1. 
\]
Let us now provide an interpretation for the statements of the lemma.
Evaluating statement~\ref{sumy_H_z_plus_deltax} for a given
type $y\in\XX$ yields the equivalent formulation
\[
\Big( {\displaystyle \sum_{x \in \XX}}
\big( \R{\tts \B}{} (z)\big)(x) \,
 \Hbar{\A}{} (z + \delta_x)\Big) (y) = \big( \Hbar{\A}{} (z) \big) (y)
 + {\displaystyle \sum_{j \in M}} 
\Big( \Big ( \Hbar{\A_{M \setminus j}}{} \otimes \R{\tts \B|^{\pa}_{A_j}}{} \Big )(z)\Big) (y).
\] 
Let us read the left-hand side as the expected number of $y$ individuals when drawing the parts of $\A$ without replacement from the population $z$ to which one individual with type  distribution  $\R{\tts \B}{} (z)$ has beed added. The statement then says that this can be achieved either by drawing \emph{all} parts of $\A$ from $z$ without replacement, \emph{or} by drawing \emph{all but one} of them from $z$ without replacement and the parts of $\B$ induced by the remaining block independently of each other and of all other blocks. \\
Likewise, evaluating statement \ref{sumy_H_z_minus_deltay} for some type $y\in \XX$ gives 
\[ \Big(\sum_{x \in \XX} \, \frac{z(x)}{N}  \Hbar{\A}{} (z - \delta_x) \Big) (y)  
 =  \frac{N - m}{N} \,  \big(\Hbar{\A}{}(z)\big) (y). \]
Let us note in passing that the left-hand side is always well-defined, since $z-\delta_x<0$ can only occur with $z(x) = 0$, in which case the term vanishes. 
This left-hand side yields the expected number of $y$ individuals when drawing the parts of $\A$ from the population $z$ \emph{after} removal of one randomly sampled individual. The statement then tells us that this is the same as \emph{first} drawing the parts of $\A$ from \emph{all} of $z$  and then deciding whether none of the $m$ affected individuals has been removed, which is the case with probability $(N - m)/N$.
\end{remark}

\begin{proof}[Proof of Lemma~$\ref{sumysumx}$]
We first observe that 
\begin{equation}\label{sum_proj}
\sum_{x \in \XX} \big  (\R{\tts \B}{} (z) \big ) (x) \ts (\pi^{}_{U} . \delta_x) 
= \pi^{}_{U} . \sum_{x \in \XX} \big (\R{\tts \B}{} (z) \big ) (x) \, \delta_x 
= \pi^{}_{U} . \big (\R{\tts \B}{} (z) \big ) 
= \R{\tts \B |^{\pa}_U}{} ( \pi^{}_U . z)
\end{equation}
by Fact~\ref{properties_R}.
We next evaluate $\sum_{x \in \XX} \big(\R{\tts \B}{} (z) \big) (x) 
\,\big [ \Rbar{\A}{} (z + \delta_x) - \Rbar{\A}{} (z) \big ]$ by
expanding $\Rbar{\A}{}$ to separate the action on $z$ from 
that on $\delta_x$,  summing against $\R{\tts \B}{} (z)$ (using~\eqref{sum_proj}),
applying Fact~\ref{R_Summe_H} and changing summation:
\begin{align*}
&  \sum_{x \in \XX}  \big ( \R{\tts \B}{} (z) \big ) (x) \,
\big [ \Rbar{\A}{} (z + \delta_x) - \Rbar{\A}{} (z) \big ] \\
& \quad\quad\quad
 = \sum_{x \in \XX}  \big (\R{\tts \B}{} (z) \big ) (x) \sum_{\varnothing \neq \underdot J \subseteq M} 
\Big ( \Rbar{\A_{M \setminus J}}{} (\pi^{}_{A_{M \setminus J}} . z) \Big ) \otimes \big( \pi^{}_{A_J} .\delta_x) \\
& \quad\quad\quad = \sum_{\varnothing \neq \underdot J \subseteq M} 
\Big( \Rbar{\A_{M \setminus J}}{} \otimes \R{\tts \B |^{\pa}_{A_J}}{} \Big) (z) = \sum_{\varnothing \neq \underdot J \subseteq M} \,
\sum_{\underdot \C \succcurlyeq \A^{}_{M \setminus J}} 
\Big (\Hbar{\tts \C}{} \otimes \R{\tts \B |^{\pa}_{A_J}}{} \Big )(z) \\
& \quad\quad\quad = \sum_{\underdot \D \succcurlyeq \A}\; \sum_{j=1}^{\lvert \D \rvert}
\Big (\Hbar{\tts \D \setminus D_j}{} \otimes \R{\tts \B |^{\pa}_{D_j}}{} \Big )(z),
\end{align*}
where, in the last step, every $A_J$ reappears as one $D_j$.
Using this together with~\eqref{H_Summe_R} and~\eqref{help_moeb_Z}, we obtain
\begin{align*}
 & \sum_{x \in \XX}  \big ( \R{\tts \B}{} (z) \big ) (x) 
\left [ \Hbar{\A}{} (z + \delta_x) - \Hbar{\A}{} (z) \right ]
  =  \sum_{\underdot \C \succcurlyeq \A}  
\mu(\A,\C) \sum_{x \in \XX}  \big ( \R{\tts \B}{} (z) \big ) (x) \left [ \Rbar{\tts \C}{} (z + \delta_x) - \Rbar{\tts \C}{} (z) \right ] \\
& \quad\quad = \sum_{\underdot \C \succcurlyeq \A}\mu(\A,\C)  
\sum_{\underdot \D \succcurlyeq \C}\; \sum_{j=1}^{\lvert \D \rvert}
\Big ( \Hbar{\tts \D \setminus D_j}{} \otimes \R{\tts \B |^{\pa}_{D_j}}{} \Big)(z) \\
&  \quad\quad = \sum_{ \underdot \D \> \A}\; \sum_{j=1}^{\lvert \D \rvert}
\Big (\Hbar{\tts \D \setminus D_j}{} \otimes \R{\tts \B |^{\pa}_{D_j}}{} \Big )(z)
 \sum_{\A \< \underdot \C \< \D }  \mu(\A,\C) 
 = \sum_{j \in M} \Big ( \Hbar{\A_{M \setminus j}}{} \otimes \R{\tts \B|^{\pa}_{A_j}}{} \Big )(z),
\end{align*}
which is statement~\ref{sumy_H_z_plus_deltax}.

In an analoguous way, we can prove statement~\ref{sumy_H_z_minus_deltay}: 
\begin{align*}
& \sum_{x \in \XX} z(x) 
\left [  \Rbar{\A}{} (z - \delta_x) - \Rbar{\A}{} (z) \right ]
  = \sum_{\varnothing \neq \underdot J \subseteq M} (-1)^{\lvert J \rvert} \ts \sum_{x \in \XX}  z(x) \,
\Big ( \Rbar{\A_{M \setminus J}}{} (\pi^{}_{A_{M \setminus J}} . z)\Big ) \otimes  
\Big ( \Rbar{\tts \1}{A_J}( \pi^{}_{A_J} .\delta_x) \Big ) \\
 & \quad\quad = \sum_{\varnothing \neq \underdot J \subseteq M} (-1)^{\lvert J \rvert} \,
 \Big (\Rbar{\A_{M \setminus J}}{} \otimes \, \Rbar{\tts \1}{A_J} \Big )(z) 
 = \sum_{\varnothing \neq \underdot J \subseteq M} (-1)^{\lvert J \rvert} \,
 \Big (\Rbar{\A_{M \setminus J} \cup A_J}{} \Big ) (z) \\
 & \quad\quad = \sum_{\varnothing \neq \underdot J \subseteq M} (-1)^{\lvert J \rvert} \,
 \sum_{\underdot \B \> \A_{M \setminus J} \cup A_J} \Hbar{\tts \B}{} (z) 
 = \sum_{\underdot \C \> \A} \Hbar{\tts \C}{} (z)\,
\sum_{j=1}^{\lvert \C \rvert}\; \sum_{ \varnothing \neq \underdot K \subseteq C_j} (-1)^{\lvert K \rvert} \\
& \quad\quad = \sum_{\underdot \C \> \A} \Hbar{\tts \C}{} (z)
\sum_{j=1}^{\lvert \C \rvert}\, \left [ (1 - 1)^{|C_j|} - 1 \right ]
= - \sum_{\underdot \C \> \A}\, \lvert \C \rvert \, \Hbar{\tts \C}{} (z),
\end{align*}
where, in the second-last step, every $A_J$ reappears as a $C_j$.
We therefore get
\begin{align*}
    \sum_{x \in \XX} z(x) 
& \left [ \Hbar{\A}{} (z - \delta_x) - \Hbar{\A}{} (z) \right ] 
  = \sum_{\underdot \B \succcurlyeq \A} \mu(\A, \B)\, \sum_{x \in \XX} 
z(x) \left [ \Rbar{\tts \B}{} (z - \delta_x) - \Rbar{\tts \B}{} (z) \right ] \\
&  \quad\quad = - \sum_{\underdot \B \succcurlyeq \A} \mu(\A, \B)\,
\sum_{\underdot \C \> \B}\, \lvert \C \rvert \, \Hbar{\tts \C}{} (z)  
 = - \sum_{\underdot \C \> \A}\, \lvert \C \rvert \, \Hbar{\tts \C}{} (z)
\sum_{\A \< \underdot \B \< \C} \mu(\A,\B) \\
& \quad\quad = - \lvert \A \rvert \, \Hbar{\A}{} (z),
\end{align*}
as claimed.
\end{proof}

We can now proceed as follows.  
\begin{proof}[Proof of Theorem~$\ref{equality_generators}$]
We start with the partitioning process. We first note that
\begin{equation}\label{weight_count}
\sumsubstack{\underdot \B \succcurlyeq \A_{M\setminus j} \cup \J}{\B|^{\pa}_{A_{M \setminus j}}=\A_{M \setminus j}}  
\frac{\big ( N - (m - 1) \big ) !}{(N - |\B|)!} = \Ncard{\J}
\end{equation}
for $j \in M$ and $|\J| \leqslant 2$. This is easily verified by direct
calculation; 
namely, for $\lvert \J \rvert =1$, the sum on the left-hand side equals
$ (N - (m - 1)) + (m - 1)=N$; for $\lvert \J \rvert =2$, it evaluates to 
\[
(N - (m - 1))\, (N - m) + (N - (m - 1)) \, (2 m - 1) + (m - 1)^2 = N^2.
\] 
We now use the formulation of the process via~\eqref{eq:theta} and~\eqref{Theta} in
the first step, normalisation and~\eqref{weight_count} in the second, Lemma~\ref{product_H_R}
in the third, and finally another normalisation step to calculate 
\begin{align}
  \sum_{\B \in \PP(S)} \varThet{\A \B}{} \, \H{\tts \B}{} (z)
& = \sum_{j \in M} \sum_{\J \in \OP_{\leqslant 2}(A_j)} 
 \frac{\r{\J}{}}{\Ncard{\J}} 
 \sumsubstack{\underdot \B \succcurlyeq \A_{M\setminus j} \cup \J}{\B|^{\pa}_{A_{M \setminus j}}=\A_{M \setminus j}}  
 \frac{\big ( N - (m - 1) \big ) !}{(N - \lvert \B \rvert)!} \, \big (\H{\tts \B}{} - \H{\A}{} \big )(z)\notag \\ 
& = \sum_{j \in M} \sum_{\J \in \OP_{\leqslant 2}(A_j)} 
   \frac{\r{\J}{}}{\Ncard{\J}} \,
 \Bigg ( \Bigg ( \sumsubstack{\underdot \B \succcurlyeq \A_{M\setminus j} \cup \J}{\B|^{\pa}_{A_{M \setminus j}}=\A_{M \setminus j}}  
 \frac{\big ( N - (m - 1) \big ) !}{N!} \; \Hbar{\tts \B}{} \Bigg ) - \Ncard{\J} \, \H{\A}{} \Bigg )(z) \notag \\ 
& = \sum_{j \in M} \sum_{\J \in \OP_{\leqslant 2}(A_j)} 
   \frac{\r{\J}{}}{\Ncard{\J}} \,
  \Bigg (  \frac{\big ( N - (m - 1) \big ) !}{N!} 
 \Big ( \Hbar{\A_{M \setminus j}}{} \otimes \, \Rbar{\J}{} \Big ) - \Ncard{\J} \, \H{\A}{} \Bigg )(z) \notag \\ 
& = \sum_{j \in M} \sum_{\J \in \OP_{\leqslant 2}(A_j) } 
   \r{\J}{}\, \Big ( \H{\A_{M \setminus j}}{} \otimes \, \R{\J}{} - \H{\A}{} \Big )(z). \label{ThetaH_last_line}
\end{align}

We now turn to the type distribution process.
Here we first evaluate, with Lemma~\ref{sumysumx}~\ref{sumy_H_z_minus_deltay}:
 \begin{align*}
 \sum_{y \in \XX} z(y) 
& \left [ \Hbar{\A}{} \big ( z + \delta_x - \delta_y \big ) 
- \Hbar{\A}{} (z) \right ] \\
&  = \sum_{y \in \XX} (z + \delta_x)(y) \, \Hbar{\A}{} \big ( (z + \delta_x) - \delta_y \big ) - \sum_{y \in \XX} (z + \delta_x)(y) \, \Hbar{\A}{} (z) \\
&  = \sum_{y \in \XX} (z + \delta_x)(y) 
\left [ \Hbar{\A}{} \big ( (z + \delta_x) - \delta_y \big ) - \Hbar{\A}{} (z + \delta_x)
+ \Hbar{\A}{} (z + \delta_x) - \Hbar{\A}{} (z) \right ] \\
&  = (N + 1 - m) \, \big [ \Hbar{\A}{} (z + \delta_x) - \Hbar{\A}{} (z) \big ] 
  - m \, \Hbar{\A}{}(z). 
 \end{align*}
From this, we obtain via summation against $\R{\tts \B}{} (z)$ and use of
Lemma~\ref{sumysumx}~\ref{sumy_H_z_plus_deltax} that
\begin{equation} 
\begin{split} 
& \sum_{x \in \XX} \sum_{y \in \XX} \big (\R{\tts \B}{} (z) \big ) (x) \, z(y) \,
\big [ \Hbar{\A}{} (z + \delta_x - \delta_y  ) - \Hbar{\A}{} (z) \big ] \\
& \quad\quad\quad  = (N + 1 - m)  \sum_{j \in M} 
\Big ( \Hbar{\A_{M \setminus j}}{} \otimes \R{\tts \B|^{\pa}_{A_j}}{} \Big )(z) 
 - m \, \Hbar{\A}{}(z). \label{sum_zy_delta_x_delta_y}
\end{split}
\end{equation}

We now have to examine $\sum_{z' \in E} \varLambda_{zz'} \H{\A}{}(z')$ for an arbitrary partition $\A$ of $S$. 
To this end, we use~\eqref{lambda} and normalisation, followed by~\eqref{sum_zy_delta_x_delta_y} and a change of summation involving~\eqref{def-marginal-rates} to calculate
\begin{align*}
 \sum_{z' \in E} \varLambda_{zz'} \, \H{\A}{}(z')
  & = \sum_{x,y \in \XX} \lambda(z;\, y,\, x) \left [ \H{\A}{}(z + \delta_x - \delta_y) - \H{\A}{}(z) \right ]  \\
  & = \frac{(N - m)!}{N!} \sum_{\B \in \OP_{\leqslant 2}(S)} \r{\,\B}{} \sum_{x,y \in \XX} \big (\R{\tts \B}{} (z)\big ) (x) \, z(y) \left [ \Hbar{\A}{} (z + \delta_x - \delta_y) - \Hbar{\A}{} (z) \right ]  \\
  & = \sum_{\B \in \OP_{\leqslant 2}(S)} \r{\,\B}{} \,\Big [ \Big ( \frac{(N - (m - 1))!}{N!} \sum_{j \in M} \Hbar{\A_{M \setminus j}}{} \otimes \, \R{\tts \B|^{\pa}_{A_j}}{} \Big ) - \frac{(N - m)!}{N!} \, m \, \Hbar{\A}{} \Big ](z) \\
  & =  \sum_{\B \in \OP_{\leqslant 2}(S)} \r{\,\B}{} \, \sum_{j \in M}  
  \Big ( \H{\A_{M \setminus j}}{} \otimes \, \R{\tts \B|^{\pa}_{A_j}}{} - \H{\A}{} \Big ) (z) \\
  & =  \sum_{j \in M} \, \sum_{\J \in \OP_{\leqslant 2}(A_j) } \, \sum_{\substack {\B \in \OP_{\leqslant 2} (S) \\ \B|^{\pa}_{A_j} = \J}} \r{\,\B}{} \,\Big ( \H{\A_{M \setminus j}}{} \otimes \, \R{\J}{} - \H{\A}{} \Big )(z)  \\
  & = \sum_{j \in M} \, \sum_{\J \in \OP_{\leqslant 2}(A_j) }  \r{\J}{} \, 
      \Big ( \H{\A_{M \setminus j}}{} \otimes \, \R{\J}{} - \H{\A}{} \Big )(z),
\end{align*}
which agrees with~\eqref{ThetaH_last_line} and proves the claim.
\end{proof}

We can now harvest some interesting consequences. First, 
Eq.~\eqref{ThetaH_last_line} 
contains a meaningful expression for the derivative:
\begin{coro} For $\A \in \PP(S)$, $z \in E$, and the population process $\{Z_t\}_{t\geqslant 0}$, we have
\[
\pushQED{\qed}
\frac{{\rm d}}{{\rm d}t}\, \E_\varphi \left[ \H{\A}{} (Z_t) \mid  Z_0 = z \right ]{ \big |_{t=0}} = \sum_{j \in M}\, \sum_{\J \in \OP_{\leqslant 2}(A_j)}  \r{\J}{\!A_j} 
 \Big ( \R{\J}{A_j} \otimes \H{\A_{M \setminus j}}{A_{M \setminus j}} 
  - \H{\A}{S} \Big ) (z) . \qedhere \popQED
\] 
\end{coro}
The right-hand side has a plausible explanation. Namely, when block $A_j$ splits
into $\J$, the other blocks in $\A$ retain their current type distribution
(namely, $\H{\A_{M \setminus j}}{} (\pi^{}_{A_{M \setminus j}} . z$)).
Independently of this, the parts of $\J$ pick their types from 
\emph{all} individuals (with replacement), including those individuals
that already carry other parts of $\A_{M \setminus j}$, which is expressed
by the tensor product with $\R{\J}{} (\pi^{}_{A_j} . z)$.

Next, via~\eqref{semigroup} together with the fact that 
$\H{\A}{}(z) = \E_\varphi\left[ \H{\A}{} (Z_t) \mid Z_t = z\right]$, 
Eqns.~\eqref{finite_duality}~and~\eqref{duality_Z_Sigma} give rise to a system of differential
equations for the expectations, namely:
\begin{coro}\label{ODE_exp} For $\A \in \PP(S)$ and the population process $\{Z_t\}_{t\geqslant 0}$, one has
\[
\pushQED{\qed}
 \frac{{\rm d}}{{\rm d}t}\, \E_\varphi \left[ \H{\A}{} (Z_t) \right ] = \sum_{\B\in \PP(S)} \varThet{\A \B}{}\; \E_\varphi\left[\H{\B}{} (Z_t ) \right]. \qedhere \popQED
\]
\end{coro}
This will be the basis for our concrete calculations in the next section.
 
 
\section{Applications and examples} \label{sec:examples}
Let us now apply our results to the cases of $n=2$ and $n=3$ sites.
Expectations will always be with respect to $\varphi$, so we will
abbreviate $\E_\varphi$ by $\E$ throughout. We will assume that
$Z_0=z$, i.e., that the initial population is deterministic.

\paragraph{Two sites.} 
 For $U = S = \{1,2\}$, with the abbreviation $r := r_{\{\{1\},\{2\}\}}$, 
the ODE system 
of Corollary~\ref{ODE_exp} reads
\begin{align}\label{H1}
	\frac{{\rm d}}{{\rm d}t} \, \E \big[ \H{\tts \{\{1,2\}\}}{}(Z_t) \big] & =  
r \, \frac{N - 1}{N} \, \E \big[ \big( \H{\tts \{\{1\},\{2\}\}}{} - \H{\tts \{\{1,2\}\}}{} \big)(Z_t) \big] \\
	\frac{{\rm d}}{{\rm d}t} \, \E \big[ \H{\tts \{\{1\},\{2\}\}}{} (Z_t) \big]  
 & = \frac{2}{N} \, \E \big[ \big( \H{\tts \{\{1,2\}\}}{} - \H{\tts \{\{1\},\{2\}\}}{} \big)(Z_t) \big], \notag
\end{align}
where we have dropped the upper index, which is always $U$.
It follows that
\begin{equation}\label{d}
  \frac{{\rm d}}{{\rm d}t} \, \E \big[ \big( \H{\tts \{\{1,2\}\}}{} - \H{\tts \{\{1\},\{2\}\}}{} \big)(Z_t) \big]  = 
- \left(\frac{2}{N} + r \, \frac{N - 1}{N} \right) 
\E \big[ \big( \H{\tts \{\{1,2\}\}}{} - \H{\tts \{\{1\},\{2\}\}}{} \big)(Z_t)\big].
\end{equation}
Since $L_{\{\{1,2\}\}} = \frac{N - 1}{N} \, ( \H{\tts \{\{1,2\}\}}{} - \H{\tts \{\{1\},\{2\}\}}{} )$, it follows
that the expected two-point LDE also decays at rate $2/N + r (N-1)/N$.
In the case of two alleles per site, an equivalent formula has appeared in
\citet[Ex.~1]{Polen2}. The corresponding result in the diffusion limit
goes back to \citet{OhtaKimura}, see also \citet[Chap.~8.2]{Durrett}. As
noted there, it may seem surprising that the correlations also decay via
resampling (even if $r = 0$); but recall that our Moran model with recombination
is an absorbing Markov chain where a single type goes to fixation in the
long run, that is, $Z_t$ will ultimately end up in a point measure.
 
The expected type distribution is now easily obtained from~\eqref{H1}~and~\eqref{d} via
\begin{align*} 
	\E \big[ \H{\tts \{\{1,2\}\}}{} (Z_t) \big] & = \E \big[ \H{\tts \{\{1,2\}\}}{} (Z_0) \big] - r \, \frac{N - 1}{N} \int_0^t \E \big[ \big( \H{\tts \{\{1,2\}\}}{} - \H{\tts \{\{1\},\{2\}\}}{} \big) (Z_\tau) \big] {\rm d} \tau \\  
 = \frac{Z_0}{N} & - \frac{r (N - 1)}{r (N - 1) + 2}  \, \Big(1 - \exp \Big(-  \frac{r (N - 1)+2}{N} \, t\Big) \Big) \E \big[ \big(\H{\tts \{\{1,2\}\}}{} - \H{\tts \{\{1\},\{2\}\}}{} \big) (Z_0) \big],
\label{solutionH2sites}
\end{align*}
where we have used that $\E \big[ \H{\tts \{\{1,2\}\}}{} (Z_0) \big]=Z_0/N$.
The asymptotic behaviour is given by
\begin{equation}\label{asymptotic} 
\lim_{t \to \infty} \E \left[\frac{Z_t}{N}\right] =  \frac{2}{2 + r (N - 1)} \, \frac{Z_0}{N}  + \frac{r (N - 1)}{2 + r (N - 1)} \, \H{\tts \{\{1\}\{2\}\}}{} (Z_0).
\end{equation}
Since $Z_t$ will ultimately absorb in a point measure, we also know that
\[
\lim_{t \to \infty} \E \left[\frac{Z_t}{N}\right] = 
\sum_{x \in \XX} \P[ Z_t \text{ absorbs  in } x] \, \delta_x,
\]
and thus
$\P[ Z_t \text{ absorbs  in } x] = \lim_{t \to \infty} \E[Z_t/N](x)$ for
all $x \in \XX$. We can therefore read off the fixation probabilities from~\eqref{asymptotic}. With probability 
$\frac{2}{2 + r (N - 1)}$ (the relative intensity of resampling),
the type that wins is drawn from the initial distribution. With probability $\frac{r (N - 1)}{2 + r (N - 1)}$ (the relative 
intensity of recombination), it is drawn from the  distribution
that results when the leading and the trailing segments are sampled from
the initial population without replacement. 

\paragraph{Three sites.} 
Now, consider $ U = S = \{1, 2, 3\}$, together with the abbreviations $r_1 = r_{\{\{1\}\{2, 3\}\}}$ and 
$r_2 = r_{\{\{1,2\}\{3\}\}}$. Let us  order the partitions of $\PP(U)$ 
as follows:
\[ \{\{1,2,3\}\}\quad \{\{1\},\{2,3\}\}\quad \{\{1,2\},\{3\}\}\quad \{\{1,3\},\{2\}\}\quad \{\{1\},\{2\},\{3\}\}.\]
The generator of the partitioning process then reads
\begin{equation} 
\varTheta = 
\left( \begin{smallmatrix}
- \frac{N - 1}{N} (r_1 + r_2) & \frac{N - 1}{N} r_1 & \frac{N - 1}{N} r_2 & 0 & 0 \\[1mm]
\frac{2}{N} - \frac{N - 1}{N^2} r_2 & - \frac{2}{N} - \frac{(N - 1)^2}{N^2} r_2 & \frac{N - 1}{N^2} r_2 & \frac{N - 1}{N^2} r_2 & \frac{(N - 1)(N - 2)}{N^2} r_2 \\[1mm]
\frac{2}{N} - \frac{N - 1}{N^2} r_1 & \frac{N - 1}{N^2} r_1 & - \frac{2}{N} - \frac{(N - 1)^2}{N^2} r_1 & \frac{N - 1}{N^2} r_1 & \frac{(N - 1)(N - 2)}{N^2} r_1 \\[1mm]
\frac{2}{N} - \frac{N - 1}{N^2} (r_1 + r_2) & \frac{N - 1}{N^2} (r_1 + r_2) & \frac{N - 1}{N^2} (r_1 + r_2) & - \frac{2}{N} - \frac{(N - 1)^2}{N^2}(r_1 + r_2) & \frac{(N - 1)(N - 2)}{N^2} (r_1 + r_2) \\[1mm]
0 & \frac{2}{N} & \frac{2}{N} & \frac{2}{N} & - \frac{6}{N}
\end{smallmatrix} \right). \label{Theta3}
\end{equation}
Recall that, by Corollary~\ref{ODE_exp}, we have 
$ \frac{\rm d}{{\rm d} t}\,  \E [ H (Z_t) ] = \varTheta \E [ H (Z_t) ]$, where $H (Z_t) := (\H{\A}{U} (Z_t))_{\A \in \PP(U)}^{}$.\\[2pt]
We now transform this system into a system in terms of correlation functions. 
Therefore, let $L (Z_t) = (\L{\A}{U} (Z_t))_{\A \in \PP(U)}^{}$. 
From~\eqref{L_in_H}, we know that $L(Z_t) = T H(Z_t)$, where the transformation matrix is given by
\[
T
= \textstyle{\frac{(N - 1) (N - 2)}{N^2}}
\left( \begin{smallmatrix}
1 & - 1 & - 1 & - 1 & 2 \\
\frac{1}{N - 2} & 1 + \frac{1}{N - 2} & \frac{-1}{N - 2} & \frac{-1}{N - 2} & - 1 \\
\frac{1}{N - 2} & - \frac{1}{N - 2} & 1 + \frac{1}{N - 2} & - \frac{1}{N - 2} & - 1 \\
\frac{1}{N - 2} & - \frac{1}{N - 2} & - \frac{1}{N - 2} & 1 + \frac{1}{N - 2} & - 1 \\
\frac{1}{(N - 1)(N - 2)} & \frac{1}{N - 2} & \frac{1}{N - 2} & \frac{1}{N - 2} & 1
\end{smallmatrix} \right).
\]
Consequently, $\frac{\rm d}{{\rm d} t}\, \E[L (Z_t)] = T \varTheta T^{-1} \, \E[L (Z_t)]$, where 
\begin{equation}
T \Theta T^{-1} = 
\left(\begin{smallmatrix}
	 - \frac{6}{N} - \frac{(N - 1) (N - 2)}{N^2} (r_1 + r_2) & 0 & 0 & 0 & 0\\
	\frac{2}{N} - \frac{(N - 1)}{N^2} (r_1 + r_2) & - \frac{2}{N} - \frac{N - 1}{N} r_2  & 0 & 0 & 0\\
	\frac{2}{N} - \frac{(N - 1)}{N^2} (r_1 + r_2) & 0 & - \frac{2}{N} - \frac{N - 1}{N} r_1 & 0 & 0\\
	\frac{2}{N} - \frac{(N - 1)}{N^2} (r_1 + r_2) & 0 & 0 & - \frac{2}{N} - \frac{N - 1}{N} (r_1 + r_2) & 0 \\
	- \frac{1}{N^2}(r_1 + r_2) & \frac{2}{N} - \frac{1}{N} r_2 & \frac{2}{N} - \frac{1}{N} r_1 & \frac{2}{N} - \frac{1}{N}(r_1 + r_2) & 0 
\end{smallmatrix} \right).\label{UTUinv}
\end{equation}
In contrast to~\eqref{Theta3}, the matrix $T \varTheta T^{-1}$ has a nice subtriangular structure, from which
we can already read off that the \emph{expected three-point LDE} 
$\E [ \L{\tts \{\{1,2,3\}\}}{}(Z_t) ]$ (cf.~\eqref{LDEexplicit}) decays exponentially according to
\[
 \frac{\rm d}{{\rm d} t} \, \E \big[ \L{\tts\{\{1,2,3\}\}}{}(Z_t) \big] = 
- \Bigg ( \frac{6 N + (N - 1) (N - 2) (r_1 + r_2)}{N^2} \Bigg ) \, \E \big[ \L{\tts\{\{1,2,3\}\}}{}(Z_t) \big].
\]
As in the case of two sites, the decay rate contains contributions from
resampling as well as from recombination. To extract more information, we recast
$T \varTheta T^{-1}$ into the diagonal form $V^{-1} T \varTheta T^{-1} V = D$,
where the entries of the diagonal matrix $D$ are those on the diagonal
of $T \varTheta T^{-1}$, i.e., its eigenvalues. Consequently, 
$ \frac{\rm d}{{\rm d} t} \, V^{-1} \, \E[L(Z_t)]= D V^{-1} \, \E[L(Z_t)]$.

With the help of  the subtriangular structure of 
$T \varTheta T^{-1}$, the matrix $V^{-1}$ can be calculated explicitly.
It is again subtriangular, but somewhat unwieldy.
To streamline the results, we now turn to the
diffusion limit, with generator $\varThetpp{}$ of Definition~\ref{margARG}.
Then $T$ and $T^{-1}$ converge to matrices $T''$ and $(T'')^{-1}_{}$, respectively, with elements $T_{\A \B}'' = \mu(\B, \A) \, \delta_{\B \< \A}^{}$ and $(T'')^{-1}_{\A \B} = \delta_{\B \< \A}^{}$, $\A,\B \in \PP(U)$ (the latter is due to inverison
from below). This yields 
\[
T'' \varThetpp{} (T'')^{-1} =
\left(\begin{smallmatrix}
	- (6 + \varrho_1 + \varrho_2) & 0 & 0 & 0 & 0\\
	2 & - (2 + \varrho_2) & 0 & 0 & 0\\
	2 & 0 & - (2 + \varrho_1) & 0 & 0\\
	2 & 0 & 0 & - (2 + \varrho_1 + \varrho_2) & 0 \\
	0 & 2 & 2 & 2 & 0 
\end{smallmatrix}\right), \label{UTUinvDiffLim}
\]
where $ \varrho_i = \lim_{N \to \infty} N r_i$, $i=1,2$. Note that the
rescaling of time has already been absorbed in $\varThetpp{}$.
In  place of $V^{-1}$, we now get
\begin{equation*}
(V'')^{-1} = 
\left( \begin{smallmatrix}
	1 & 0 & 0 & 0 & 0 \\
	\frac{ 2}{(2 + \varrho_2) (4 + \varrho_1)} & \frac{ 1}{2 + \varrho_2} & 0 & 0 & 0 \\
	\frac{ 2}{(2 + \varrho_1) (4 + \varrho_2)} & 0 & \frac{ 1}{2 + \varrho_1} & 0 & 0 \\
	\frac{ 1}{2 (2 + \varrho_1 + \varrho_2)} & 0 & 0 & \frac{ 1}{2 + \varrho_1 + \varrho_2} & 0 \\
	\frac{4 (\varrho_1 \varrho_2 + (2 + \varrho_1 + \varrho_2) (6 + \varrho_1 + \varrho_2))}{(2 + \varrho_1) (2 + \varrho_2) (2 + \varrho_1 + \varrho_2) (6 + \varrho_1 + \varrho_2)} & \frac{2}{2 + \varrho_2} & \frac{2}{2 + \varrho_1} & \frac{2}{2 + \varrho_1 + \varrho_2} & 1
\end{smallmatrix} \right), \label{VinvDiffLim}
\end{equation*}
which diagonalises $T'' \varThetpp{} {(T'')}^{-1}$. 
This shows that, in contrast to $\lvert U\rvert =2$, the linear combinations of $\E[\L{\A}{}(Z_t)]$'s that
decay exponentially have coefficients  
depending on the recombination rates (with exception of 
$\E [ \L{\tts\{\{1,2,3\}\}}{} (Z_t)]$). As an example,
$(4 + \varrho^{}_1)\, \E [ \L{\tts\{\{1\}\{2,3\}\}}{} (Z_t)] + 2 \, \E [\L{\tts\{\{1,2,3\}\}}{} (Z_t)]$ is one such combination and decays at rate
$2 + \varrho^{}_2$. 
Solution of the complete system is still possible due to the
triangular structure; however, it is somewhat tedious since it involves
the linear combination given in the last line of $(V'')^{-1}$. Further
progress may be possible if alternative scalings are employed, 
such as the  \emph{loose linkage approach} suggested recently by 
\citet{JenkinsFearnheadSong}.


\section{Conclusion}
Let us summarise our findings. We have described a marginal
ancestral recombination process (ARP) and proved a duality result that relates
the ARP with the Moran model forward in time, via so-called sampling
functions. This was achieved by extending the recombinator formalism,
which had previously proved useful in the context of deterministic
recombination equations, to the stochastic setting. The ARP, together
with the duality result, reveals the genealogical structure hidden in
the work of \citet{Polen}, who approached the matter by functional-analytic means and forward in time. 
It also leads to an explicit and closed system of ordinary differential equations for the expected
sampling functions, from which the expected linkage disequilibria
of all orders can be calculated. It is quite remarkable that such a closed ODE system
exists: after all, the sampling functions are nonlinear, and the
attempt to write down the differential equation for the expectation of
a nonlinear quantity usually results in a hierarchy of equations that
does not close; see \citet{BaakeHustedt} for more on the moment closure problem
in the case of recombination. We would like to emphasise
that the favourable structure is due to the marginalisation, 
which gives efficient access to correlation functions, but not to variances,
for example.

Unlike \citet{Polen}, we have not included mutation so far. However, since
mutation acts independently of recombination, it should be straightforward to
superimpose it on the population process as well as the partitioning process. It will be
rewarding to study the interplay of mutation (which increases LDE)
with recombination and resampling (which decrease LDE) within the
framework established here.

\subsection*{Acknowledgements}
It is our pleasure to thank Noemi Kurt, Cristian Giardina, and
Frank Redig for a primer to duality theory, Fernando Cordero for
helpful discussions, and Michael Baake for
his help to improve the manuscript. The authors gratefully acknowledge
the support from the Priority Programme \emph{Probabilistic Structures in
Evolution (SPP 1590)}, which is funded by Deutsche Forschungsgemeinschaft
(German Research Foundation, DFG).

\bibliographystyle{plainnat} 
\renewcommand{\bibfont}{\small}

%

\end{document}